\documentclass{article}%
\usepackage{graphicx}
\usepackage{amsmath}
\usepackage{amsfonts}
\usepackage{amssymb}
\usepackage{setspace}
\usepackage[hmargin=3.5cm,vmargin=3.5cm]{geometry}%
\providecommand{\U}[1]{\protect\rule{.1in}{.1in}}
\providecommand{\U}[1]{\protect\rule{.1in}{.1in}}
\providecommand{\U}[1]{\protect\rule{.1in}{.1in}}
\providecommand{\U}[1]{\protect\rule{.1in}{.1in}}
\providecommand{\U}[1]{\protect\rule{.1in}{.1in}}
\providecommand{\U}[1]{\protect\rule{.1in}{.1in}}
\providecommand{\U}[1]{\protect\rule{.1in}{.1in}}
\providecommand{\U}[1]{\protect\rule{.1in}{.1in}}
\providecommand{\U}[1]{\protect\rule{.1in}{.1in}}
\providecommand{\U}[1]{\protect\rule{.1in}{.1in}}
\providecommand{\U}[1]{\protect\rule{.1in}{.1in}}
\providecommand{\U}[1]{\protect\rule{.1in}{.1in}}
\providecommand{\U}[1]{\protect\rule{.1in}{.1in}}
\providecommand{\U}[1]{\protect\rule{.1in}{.1in}}
\providecommand{\U}[1]{\protect\rule{.1in}{.1in}}
\providecommand{\U}[1]{\protect\rule{.1in}{.1in}}
\providecommand{\U}[1]{\protect\rule{.1in}{.1in}}
\providecommand{\U}[1]{\protect\rule{.1in}{.1in}}
\providecommand{\U}[1]{\protect\rule{.1in}{.1in}}
\providecommand{\U}[1]{\protect\rule{.1in}{.1in}}
\providecommand{\U}[1]{\protect\rule{.1in}{.1in}}
\providecommand{\U}[1]{\protect\rule{.1in}{.1in}}
\providecommand{\U}[1]{\protect\rule{.1in}{.1in}}
\providecommand{\U}[1]{\protect\rule{.1in}{.1in}}
\providecommand{\U}[1]{\protect\rule{.1in}{.1in}}
\providecommand{\U}[1]{\protect\rule{.1in}{.1in}}
\providecommand{\U}[1]{\protect\rule{.1in}{.1in}}
\providecommand{\U}[1]{\protect\rule{.1in}{.1in}}

\newtheorem{theorem}{Theorem}
{}

\newtheorem{definition}{Definition}

\newtheorem{lemma}{Lemma}
{}

\newtheorem{remark}{Remark}

\newenvironment{proof}[1][Proof]{\textbf{#1.} }{\ \rule{0.5em}{0.5em}}

\begin{document}

\title{Asymptotic \ Analysis of Non-self-adjoint Hill Operators }
\author{O. A. Veliev\\{\small Depart. of Math., Dogus University, Ac\i badem, Kadik\"{o}y, \ }\\{\small Istanbul, Turkey.}\ {\small e-mail: oveliev@dogus.edu.tr}}
\date{}
\maketitle

\begin{abstract}
We obtain the uniform asymptotic formulas for the eigenvalues and
eigenfunctions of the Sturm-Liouville operators $L_{t}(q)$ with a potential
$q\in L_{1}[0,1]$ and with $t-$periodic boundary conditions, $t\in(-\pi,\pi]$.
Using these formulas, we find sufficient conditions on the potential $q$ such
that the number of spectral singularities in the spectrum of the Hill
operator$\ L(q)$ in $L_{2}(-\infty,\infty)$ is finite. Then we prove that the
operator $L(q)$ has no spectral singularities at infinity and it is an
asymptotically spectral operator provided that the potential$\ q$ satisfies
the sufficient conditions.



\end{abstract}

\section{\bigskip Introduction and Preliminary Facts}

Let $L(q)$ be the Hill operator generated in $L_{2}(-\infty,\infty)$ by the expression%

\begin{equation}
-y^{^{\prime\prime}}+q(x)y,
\end{equation}
where $q(x)$ is a complex-valued summable function on $[0,1]$ and
$q(x+1)=q(x)$ for a.e. $x\in(-\infty,\infty)$. It is well-known that (see [7],
[23] for real and [5], [16]-[18] for complex-valued $q$) the spectrum
$S(L(q))$ of the operator $L(q)$ is the union of the spectra $S(L_{t}(q))$ of
the Sturm-Liouville operators $L_{t}(q)$ for $t\in(-\pi,\pi],$ where
$L_{t}(q)$ is the operator generated in $L_{2}[0,1]$ by (1) and by the
boundary conditions
\begin{equation}
y(1)=e^{it}y(0),\text{ }y^{^{\prime}}(1)=e^{it}y^{^{\prime}}(0).
\end{equation}

In this paper we obtain the asymptotic formulas, uniform with respect to
$t\in(-\pi,\pi],$ for the\ eigenvalues and eigenfunctions of $L_{t}(q)$. (We
recall that the formula\ $f(k,t)=O(h(k))$ is said to be uniform with respect
to $t$ in a set $A$ if there exist positive constants $M$ and $N$ such that
$\mid f(k,t))\mid<M\mid h(k)\mid$ for all $t\in A$ and $\mid k\mid\geq N.)$
Using these asymptotic formulas, we find sufficient conditions on the
potential $q$ such that the number of the spectral singularities in $S(L(q))$
is finite and in a certain sense $L(q)$ is an asymptotically spectral operator.

The spectral expansion for the self-adjoint operator $L(q)$ was constructed by
Gelfand [7] and Titchmarsh [23]. Tkachenko [24] proved that the
non-self-adjoint operator $L(q)$ can be reduced to the triangular form if all
eigenvalues of the operators $L_{t}(q)$ for all $t\in(-\pi,\pi]$ are simple.
McGarvey [17] proved that $L(q)$ is a spectral operator if and only if the
projections of the operators $L_{t}(q)$ are bounded uniformly with respect to
$t$ in $(-\pi,\pi]$. However, in general, the eigenvalues of $L_{t}(q)$ are
not simple and their projections are not uniformly bounded. For instance,
Gasymov [6] investigated the operator $L(q)$ with the potential$\ q$ which can
be continued analytically onto upper half plane and proved that this operator
(in particular $L(q)$ with the simple potential $q(x)=e^{i2\pi x}),$ has
infinitely many spectral singularities. Note that the spectral singularities
of the operator $L(q)$ are the points of $S(L(q))$ in neighborhoods of which
the projections of $L(q)$ are not uniformly bounded. In [26] we proved that a
number $\lambda=\lambda_{n}(t)\in S(L)$ is a spectral singularity of $L(q)$ if
and only if the operator $L_{t}(q)$ has an associated function at the point
$\lambda_{n}(t).$ In [25] (see also [27]) we constructed the spectral
expansion for the operator $L(q)$ with a continuous and complex-valued
potential. In [28], we obtained the asymptotic formulas for the eigenvalue and
eigenfunction of $L_{t}(q)$ with $q\in L_{1}[0,1]$ and $t\neq0,\pi.$ Then
using these formulas, we proved that the eigenfunctions and associated
functions of $L_{t\text{ }}$form a Riesz basis in $L_{2}[0,1]$ for $t\neq
0,\pi$ and constructed the spectral expansion for the operator $L(q).$ (See
also [13], [29], [30] for the spectral expansion of the differential operators
with periodic coefficients). Recently, Gesztezy and Tkachenko [8], [9] proved
two versions of a criterion for the Hill operator $L(q)$ with $q\in
L_{2}[0,1]$ to be a spectral operator of scalar type, one analytic and one
geometric. The analytic version was stated in term of the solutions of the
Hill's equation. The geometric version of the criterion used the algebraic and
geometric \ properties of the spectra of the periodic/antiperiodic and
Dirichlet boundary value problems.

Since the spectral property of $L(q)$ is strongly connected with the operators
$L_{t}(q)$ for $t\in(-\pi,\pi],$ let us discuss briefly the works devoted to
$L_{t}(q).$ It is known that the operator $L_{t}(q)$ is Birkhoff regular [14].
In the case $t\neq0,\pi$ it is strongly regular and the root functions of the
operator $L_{t}(q)$ form a Riesz basis (this result was proved independently
in [4], [12] and [19]). In the cases $t=0$ and $t=\pi,$ the operator
$L_{t}(q)$ is not strongly regular. In the case when an operator is regular
but not strongly regular the root functions generally do not form even usual
basis. However, it is known [20], [21] that they can be combined in pairs, so
that the corresponding 2-dimensional subspaces form a Riesz basis of subspaces.

Let us also briefly describe some historical developments related to the Riesz
basis property of the root functions of the periodic and antiperiodic boundary
value problems. We will focus only on the periodic problem. The antiperiodic
problem is similar to the periodic one. In 1996 at a seminar in MSU Shkalikov
formulated the following result. Assume that $q(x)$ is a smooth potential,
\begin{equation}
q^{(k)}(0)=q^{(k)}(1),\quad\forall\,k=0,1,...,s-1
\end{equation}
and $q^{(s)}(0)\neq q^{(s)}(1)$. Then the root functions of the operator
$L_{0}(q)$ form a Riesz basis in $L_{2}[0,1]$. Kerimov and Mamedov [11]
obtained the rigorous proof of this result in the case $q\in C^{4}%
[0,1],\ q(1)\neq q(0)$. Indeed, this result remains valid for an arbitrary
$s\geq0$ and it was obtained in Corollary~2 of [22].

Another approach is due to Dernek and Veliev [1]. The result was obtained in
terms of the Fourier coefficients of the potential $q.$ Namely, we proved that
if the conditions
\begin{align}
\lim_{n\rightarrow\infty}\frac{\ln\left\vert n\right\vert }{nq_{2n}}  &
=0,\text{ }\\
q_{2n}  &  \sim q_{-2n}%
\end{align}
hold, then the root functions of $L_{0}(q)$ form a Riesz basis in $L_{2}%
[0,1]$, where $q_{n}=:(q,e^{i2\pi nx})$ is the Fourier coefficient of $q$ and
everywhere, without loss of generality, it is assumed that $q_{0}=0.$ Here
$(.,.)$ denotes inner product in $L_{2}[0,1]$ and $a_{n}\sim b_{n}$ means that
$a_{n}=O(b_{n})$ and $b_{n}=O(a_{n})$ as $\ n\rightarrow\infty.$ Makin [15]
improved this result. Using another method he proved that the assertion on the
Riesz basis property remains valid if condition (5) holds, but condition (4)
is replaced by a less restrictive one: $q\in W_{1}^{s}[0,1],$ (3) holds and
$\mid q_{2n}\mid>c_{0}n^{-s-1}$ with some$\ \,c_{0}>0$ for sufficiently large
$n,$ where $s$ is a nonnegative integer. Besides, some conditions which imply
the absence of the Riesz basis property were presented in [15]. \ The results
which we obtained in [22] are more general and cover all the results discussed
above. Several theorems on the Riesz basis property of the root functions of
the operator $L_{0}(q)$ were proved. One of the main results of [22] is the following:

\textit{ }Let $p\geq0$ be an arbitrary integer, $q\in W_{1}^{p}[0,1]$ and (3)
holds with some $s\leq p.$ Suppose that there is a number $\varepsilon>0$ such
that either the estimate
\begin{equation}
|q_{2n}-S_{2n}+2Q_{0}Q_{2n}|\geq\varepsilon n^{-s-2}%
\end{equation}
or the estimate
\begin{equation}
|q_{-2n}-S_{-2n}+2Q_{0}Q_{-2n}|\geq\varepsilon n^{-s-2}%
\end{equation}
hold, where $Q_{k}=(Q(x),\,e^{2\pi ikx})$ and$\ \,S_{k}=(S(x),\,e^{2\pi ikx})$
are the Fourier coefficients of
\[
Q(x)=\int_{0}^{x}q(t)\,dt\text{ and}\quad S(x)=Q^{2}(x).
\]
Then the condition
\begin{equation}
q_{2n}-S_{2n}+2Q_{0}Q_{2n}\sim q_{-2n}-S_{-2n}+2Q_{0}Q_{-2n}%
\end{equation}
is necessary and sufficient for the root functions of $L_{0}(q)$ to form a
Riesz basis. Moreover, if (6) (or (7)) and (8) hold then all the large
eigenvalues of $L_{0}(q)$ are simple.

Some sharp results on the absence of the Riesz basis property were obtained by
Djakov and Mitjagin [2]. Moreover, recently, \ Djakov and Mitjagin\ [3]
obtained some interesting results about the Riesz basis property of the root
functions of the operators $L_{0}(q)$ with trigonometric polynomial
potentials. Here we do not formulate precisely the results of [2] and [3],
since it will take some additional pages which are not related to our results.
Very recently Gesztezy and Tkachenko [10] proved a criterion \textit{ }for the
root functions of $L_{0}(q)$ to form a Riesz basis in terms of the spectra of
the periodic and Dirichlet boundary value problems.

Next, we present some preliminary facts, from [28] and [1], which are needed
in the following.

\textbf{Result 1 (see [28]).}\textit{ The eigenvalues }$\lambda_{n}%
(t)$\textit{ \ and the eigenfunctions }$\Psi_{n,t}(x)$\textit{ of the operator
}$L_{t}(q)$\textit{ for }$t\neq0,\pi,$\textit{ satisfy the following
asymptotic formulas }%
\begin{equation}
\lambda_{n}(t)=(2\pi n+t)^{2}+O(\frac{ln\left\vert n\right\vert }{n}),\text{
}\Psi_{n,t}(x)=e^{i(2\pi n+t)x}+O(\frac{1}{n}).
\end{equation}
\textit{These asymptotic formulas are uniform with respect to }$t$\textit{ in
}$[\rho,\pi-\rho],$ \textit{where }$\rho\in(0,\frac{\pi}{2})$ \textit{(see
Theorem 2 of [9]). In other words, there exist positive numbers }$N(\rho
)$\textit{ and }$M(\rho),$\textit{ independent of }$t,$\textit{ such that the
eigenvalues }$\lambda_{n}(t)$\textit{ for \ }$t\in\lbrack\rho,\pi-\rho
]$\textit{ and}$\mid n\mid>N(\rho)$\textit{ are simple and the terms }%
$O(\frac{1}{n})$\textit{, }$O(\frac{ln\left\vert n\right\vert }{n})$\textit{
in (9) do not depend on }$t.$

\textbf{Result 2 (see [1])}\textit{. Let conditions (4) and (5) hold. Then:}

$(a)$\textit{ All sufficiently large eigenvalues of the operator }$L_{0}%
(q)$\textit{ are simple. They consist of two sequences \ }$\{\lambda
_{n,1}:n>N_{0}\}$\textit{ and }$\{\lambda_{n,2}:n>N_{0}\}$\textit{\ satisfying
}%
\begin{equation}
\lambda_{n,j}=(2\pi n)^{2}+(-1)^{j}p_{2n}+O\left(  \frac{ln\left\vert
n\right\vert }{n}\right)
\end{equation}
\textit{for }$j=1,2,$\textit{ where }$p_{n}=(q_{n}q_{-n})^{\frac{1}{2}}%
.$\textit{ The corresponding eigenfunctions }$\varphi_{n,j}(x)$\textit{
satisfy }%
\begin{equation}
\varphi_{n,j}(x)=e^{i2\pi nx}+\alpha_{n,j}e^{-i2\pi nx}+O(\frac{1}{n}),
\end{equation}
\textit{where }$\alpha_{n,j}\sim1,$ $\alpha_{n,j}=\frac{(-1)^{j}p_{2n}}%
{q_{2n}}+O\left(  \frac{ln\left\vert n\right\vert }{nq_{2n}}\right)  ,$
$j=1,2.$

$(b)$\textit{\ The root functions of }$L_{0}(q)$\textit{ form a Riesz basis in
}$L_{2}[0,1].$

In [31] and [32] we generalized the results of [1] for the operators generated
by the differential equation of order $n>2$ and by the system of differential equations.

To summarize, in [1] and [22] we obtained the asymptotic formulas for the
operators $L_{t}(q)$ with $t=0,\pi$. In [28] we obtained the asymptotic
formulas for the operators $L_{t}(q)$ which are uniform with respect to
$t\in\lbrack\rho,\pi-\rho].$ In this paper, we obtain the uniform asymptotic
formulas in much more complicated case of $t\in\lbrack0,\rho]\cup\lbrack
\pi-\rho,\pi]$ (see Theorem 3 and 4). We note that in our description, some
formulas of Section 2 are similar to those given in [1], [22] and [28], but
here we wish to obtain the uniform, with respect to $t\in\lbrack0,\rho
]\cup\lbrack\pi-\rho,\pi],$ formulas which are absent in these papers. We will
focus only on the case $t\in\lbrack0,\rho].$ (The case $t\in\lbrack\pi
-\rho,\pi]$ can be considered in the same way).\ Since the eigenvalues of
$L_{-t}(q)$ coincide with those of $L_{t}(q)$, we obtain the uniform, with
respect to $t$ in $(-\pi,\pi],$ asymptotic formulas for the operators
$L_{t}(q)$. These formulas imply that if the potential $q$ satisfies certain
conditions, then there exists a positive constant $C$ independent of $t$ such
that all the eigenvalues of $L_{t}(q)$ lying outside the disk $\{\lambda
\in\mathbb{C}:\left\vert \lambda\right\vert \leq C\}$ are simple for all the
values of $t$ in $(-\pi,\pi]$. Since the spectral singularities of the
operator $L(q)$ are contained in the set of multiple eigenvalues of
$L_{t}(q),$ we obtain the sufficient conditions on $q$ such that the Hill
operator $L(q)$ has at most finitely many spectral singularities. Moreover, we
prove that if $q$ satisfies these conditions then $L(q)$ has no spectral
singularity at infinity and in the sense of Definition 3 given in Section 3,
the operator $L(q)$ is an asymptotically spectral operator.

\section{Uniform Asymptotic Formulas for $L_{t}(q)$}

It is well-known that the eigenvalues of $L_{t}(q)$ are the squares of the
roots of the equation
\begin{equation}
F(\xi)=2\cos t,
\end{equation}
where $F(\xi)=\varphi^{^{\prime}}(1,\xi)+\theta(1,\xi),$ and $\varphi(x,\xi)$
and $\theta(x,\xi)$ are the solutions of the equation%

\[
-y^{^{\prime\prime}}+q(x)y=\xi^{2}y
\]
satisfying the initial conditions $\theta(0,\xi)=\varphi^{^{\prime}}%
(0,\xi)=1,\quad\theta^{^{\prime}}(0,\xi)=\varphi(0,\xi)=0.$ In [14] (see
chapter 1, sec. 3) it was proved that
\begin{equation}
F(\xi)-2\cos\xi=e^{\left\vert Im\xi\right\vert }\varepsilon(\xi),\text{
}\underset{\left\vert \xi\right\vert \rightarrow\infty}{\lim}\varepsilon
(\xi)=0.
\end{equation}

Let us consider the functions $F(\xi)-2\cos\xi$ and $2\cos\xi-\cos t$ on the
circle
\begin{equation}
C(n,t,\rho)=:\{\xi\in\mathbb{C}:\left\vert \xi-(2\pi n+t)\right\vert =3\rho\},
\end{equation}
where $t\in\lbrack0,\rho]$ and $\rho$ is a sufficiently small fixed number. By
(13) there exists a positive number $N(0,\rho)$ such that
\begin{equation}
\left\vert F(\xi)-2\cos\xi\right\vert <\rho^{2}%
\end{equation}
for $\xi\in C(n,t,\rho)$ whenever $n>N(0,\rho)$ and $t\in\lbrack0,\rho].$ On
the other hand, using the Taylor formula of $\cos\xi$ at the point $2\pi n+t$
for $\xi=2\pi n+t+3\rho e^{i\alpha},$ where $\alpha\in(-\pi,\pi],$ and taking
into account the inequalities $\left\vert \sin t\right\vert \leq\rho$ and
$\left\vert \cos t\right\vert >\frac{9}{10}$ for $t\in\lbrack0,\rho],$ we
obtain
\begin{equation}
\mid2\cos\xi-2\cos t\mid=2\mid-3\rho e^{i\alpha}\sin t+\frac{9}{2}\rho
^{2}e^{2i\alpha}\cos t+O(\rho^{4})\mid>2\rho^{2}.
\end{equation}
By the Rouche's theorem, it follows from (15) and (16) that equation (12) and
\begin{equation}
\cos\xi-\cos t=0
\end{equation}
have the same number of the roots inside $C(n,t,\rho),$ where $n>N(0,\rho).$
Since equation (17) has $2$ roots inside the circle $C(n,t,\rho),$ equation
(12) has also $\ 2$ roots (counting multiplicity) inside this circle for $n>$
$N(0,\rho)$. On the other hand, it is proved in [14] (see chapter 1, sec. 3)
that the estimation
\[
F(\xi)-2\cos\xi=o(\cos\xi-\cos t)
\]
holds on the boundaries of the admissible strip $K_{n}=:\left\{
\xi:\left\vert \operatorname{Re}\xi\right\vert <(2n+1)\pi\right\}  $ for
$t\in\lbrack0,\rho].$ Hence the number of the roots of equations (12) and (17)
are the same in the strip $K_{n}.$ Similarly, these equations have the same
number of the roots in the set $K_{n+1}\backslash K_{n}$ for large $n$. The
following remark follows from these arguments.

\begin{remark}
There exists a large number $N(0,\rho)$ such that the number of the roots of
equations (12) lying in the strip $K_{N}$ is $2N+1.$ Denote these roots by
$\xi_{n}(t)$ for

$n=0,\pm1,\pm2,...,\pm N.$ The roots of equation (12) lying outside $K_{N}$
consist of the roots lying inside the contours $C(n,t,\rho),$ defined in (14),
for $n>$ $N(0,\rho).$ Moreover, (12) has two roots, denoted by $\xi_{n,1}(t)$
and $\xi_{n,2}(t),$ lying inside $C(n,t,\rho)$. Thus
\begin{equation}
\left\vert \xi_{n,j}(t)-(2\pi n+t)\right\vert <3\rho,\text{ }\forall\left\vert
n\right\vert >N(0,\rho),\text{ }t\in\lbrack0,\rho],\text{ }j=1,2.
\end{equation}
Since the entire function $\frac{dF}{d\xi}$ has a finite number of zeros
inside the circle

$\{\xi\in\mathbb{C}:\left\vert \xi-2\pi n\right\vert =4\rho\}$ and this circle
encloses $C(n,t,\rho)$ for all $t\in\lbrack0,\rho]$, there exist at most
finite $t_{1},t_{2},...,t_{k}$ from $(0,\rho)$ for which $\xi_{n}(t_{k})$ is a
double root of (12). Let $0<t_{1}<t_{2}<...<t_{k}<\rho.$ By the implicit
function theorem the functions $\xi_{n,1}(t)$ and $\xi_{n,2}(t)$ can be chosen
as analytic in intervals $(0,t_{1}),$ $(t_{k},\rho)$ and $(t_{s},t_{s+1})$ for
$s=1,2,...,k-1.$ Let $\xi$ be any limit point of $\xi_{n,j}(t)$ as
$t\rightarrow t_{s}.$ Since $F(\xi_{n,j}(t))=2\cos t$ for $j=1,2$ and $F$ is
continuous, we have $F(\xi)=2\cos t_{s}.$ However, this equation has only one
double root $\xi_{n,1}(t_{s})=\xi_{n,2}(t_{s})$ inside $C(n,t_{s},\rho).$
Thus
\[
\lim_{t\rightarrow t_{s}^{-}}\xi_{n,1}(t)=\lim_{t\rightarrow t_{s}^{+}}%
\xi_{n,1}(t)=\lim_{t\rightarrow t_{s}^{-}}\xi_{n,2}(t)=\lim_{t\rightarrow
t_{s}^{+}}\xi_{n,1}(t)=\xi_{n,1}(t_{s})=\xi_{n,2}(t_{s})
\]
for $s=1,2,...,k.$ This implies that the eigenvalues $\lambda_{n,1}(t)=$
$\xi_{n,1}^{2}(t)$ and $\lambda_{n,2}(t)=$ $\xi_{n,2}^{2}(t)$ of $L_{t}(q)$
can be chosen as continuous function on $(0,\rho).$ By the result of [28] (see
introduction)\ $\lambda_{n,1}(\rho)$ and $\lambda_{n,2}(\rho)$ are the simple
eigenvalues of $L_{\rho}(q)$ for $n>N(\rho).$ Moreover, if \textit{ }$q\in
L_{1}[0,1]$ \textit{and }(4), (5) hold then by the result of [1]
$\lambda_{n,1}(0)$ and $\lambda_{n,2}(0)$ are the simple eigenvalues of
$L_{0}$ for $n>N_{0}.$ These arguments imply the continuity of the functions
$\lambda_{n,1}(t),$ $\lambda_{n,2}(t)$ and
\begin{equation}
d_{n}(t)=:\left\vert \lambda_{n,1}(t)-\lambda_{n,2}(t)\right\vert
\end{equation}
on $[0,\rho]$ for $n>N=:\max\{N(0,\rho),N(\rho),N_{0}\}$. By (18) we have
\begin{equation}
\left\vert \lambda_{n,j}(t)-(2\pi n+t)^{2}\right\vert <15\pi n\rho
\end{equation}
for $t\in\lbrack0,\rho],$ $n>N$ and $j=1,2$. Thus for $t\in\lbrack0,\rho]$ and
$n>N$ the disk
\begin{equation}
D(n,t,\rho)=:\{\lambda\in\mathbb{C}:\left\vert \lambda-(2\pi n+t)^{2}%
\right\vert <15\pi n\rho\}
\end{equation}
contains two eigenvalues (counting multiplicity) $\lambda_{n,1}(t)$ and
$\lambda_{n,2}(t)$ that are continuous function on the interval $[0,\rho].$ In
addition to these eigenvalues, the operator $L_{t}(q)$ for $t\in\lbrack
0,\rho]$ has only $2N+1$ eigenvalues.
\end{remark}

Using (20), one can readily see that
\begin{equation}
\left\vert \lambda_{n,j}(t)-(2\pi(n-k)+t)^{2}\right\vert >\left\vert
k\right\vert \left\vert 2n-k\right\vert
\end{equation}
for $k\neq0,2n$ and $t\in\lbrack0,\rho]$, where $n>N$ and $j=1,2.$ To obtain
the uniform asymptotic formulas for the eigenvalues $\lambda_{n,j}(t)$ and
normalized eigenfunctions $\Psi_{n,j,t}(x)$, we use (22) and the iteration of
the formula
\begin{equation}
(\lambda_{n,j}(t)-(2\pi(n-k)+t)^{2})(\Psi_{n,j,t},e^{i(2\pi(n-k)+t)x}%
)=(q\Psi_{n,j,t},e^{i(2\pi(n-k)+t)x})
\end{equation}
which can be obtained obtained from $-\Psi_{n,j,t}^{^{\prime\prime}}%
+q\Psi_{n,j,t}=\lambda_{n,j}(t)\Psi_{n,j,t}$ by multiplying $e^{i(2\pi
(n-k)+t)x}$. To iterate (23) we use the following lemma.

\begin{lemma}
For the right-hand side of (23) the following equality
\begin{equation}
(q\Psi_{n,j,t},e^{i(2\pi(n-k)+t)x})=\sum_{m=-\infty}^{\infty}q_{m}%
(\Psi_{n,j,t},e^{i(2\pi(n-k-m)+t)x})
\end{equation}
and inequality%
\begin{equation}
\left\vert (q\Psi_{n,j,t},e^{i(2\pi(n-k)+t)x})\right\vert <3M
\end{equation}
hold for all $n>N,$ $k\in\mathbb{Z},$ $j=1,2$ and $t\in\lbrack0,\rho],$ where
$M=\sup_{n\in\mathbb{Z}}\left\vert q_{n}\right\vert ,$ and $N$ \ is defined in
Remark 1. The eigenfunction $\Psi_{n,j,t}(x)$ satisfies the following, uniform
with respect to $t\in\lbrack0,\rho],$ asymptotic formula
\begin{equation}
\Psi_{n,j,t}(x)=u_{n,j}(t)e^{i(2\pi n+t)x}+v_{n,j}(t)e^{i(-2\pi n+t)x}%
+h_{n,j,t}(x),
\end{equation}
where $u_{n,j}(t)=(\Psi_{n,j,t}(x),e^{i(2\pi n+t)x}),$ $v_{n,j}(t)=(\Psi
_{n,j,t}(x),e^{i(-2\pi n+t)x}),$
\begin{equation}
(h_{n,j,t},e^{i(\pm2\pi n+t)x})=0,\text{ }\left\Vert h_{n,j,t}\right\Vert
=O(\frac{1}{n}),\text{ }\sup_{x\in\lbrack0,1],\text{ }t\in\lbrack0,\rho]}\mid
h_{n,j,t}(x)\mid=O\left(  \frac{ln\left\vert n\right\vert }{n}\right)  ,
\end{equation}%
\begin{equation}
\left\vert u_{n,j}(t)\right\vert ^{2}+\left\vert v_{n,j}(t)\right\vert
^{2}=1+O(\frac{1}{n^{2}}).
\end{equation}

\end{lemma}

\begin{proof}
Equality (24) is obvious for $q\in L_{2}[0,1].$ For $q\in L_{1}[0,1]$\ see
Lemma 1 of [28]. Since $q\Psi_{n,j,t}\in L_{1}[0,1],$ we have
\[
\lim_{\left\vert m\right\vert \rightarrow\infty}(q\Psi_{n,j,t},e^{i(2\pi
(n-k-m)+t)x})=0.
\]
Therefore there exist $C(t)$ and $k_{0}(t)$ such that
\[
\underset{s\in\mathbb{Z}}{max}\left\vert (q\Psi_{n,j,t},e^{i(2\pi
s+t)x})\right\vert =\left\vert (q\Psi_{n,j,t},e^{i(2\pi(n-k_{0})+t)x}%
)\right\vert =C(t).
\]
Now, using (22)-(24) and the obvious relations
\begin{equation}
\left\vert q_{m}\right\vert \leq M,\text{ }\sum_{k\neq0,2n}\frac{1}{\mid
k(2n-k)\mid}=O\left(  \frac{\ln n}{n}\right)
\end{equation}
for $m\in\mathbb{Z},$ we obtain
\[
C(t)=\left\vert (q\Psi_{n,j,t},e^{i(2\pi(n-k_{0})+t)x})\right\vert =\mid
\sum_{m=-\infty}^{\infty}q_{m}(\Psi_{n,j,t},e^{i(2\pi(n-k_{0}-m)+t)x})\mid=
\]%
\[
\mid q_{-k_{0}}(\Psi_{n,j,t},e^{i(2\pi n+t)x})+q_{2n-k_{0}}(\Psi
_{n,j,t},e^{i(-2\pi n+t)x})\mid+
\]%
\begin{align*}
&  \mid\sum_{m\neq-k_{0},2n-k_{0}}q_{m}\frac{(q\Psi_{n,j,t},e^{i(2\pi
(n-k_{0}-m)+t)x})}{\lambda_{n,j}(t)-(2\pi(n-k_{0}-m)+t)^{2}}\mid\leq2M+\\
&  \sum_{m\neq-k_{0},2n-k_{0}}\frac{MC(t)}{\mid\lambda_{n,j}(t)-(2\pi
(n-k_{0}-m)+t)^{2}\mid}=2M+C(t)O\left(  \frac{\ln n}{n}\right)
\end{align*}
which implies that $C(t)<3M$ for all $t\in\lbrack0,\rho]$. Inequality (25) is
proved. This with (23), (22) \ and (29) gives
\[
\sum_{k\neq\pm n}\left\vert (\Psi_{n,j,t},e^{i(2\pi k+t)x})\right\vert
=O\left(  \frac{\ln n}{n}\right)  ,\text{ }\sum_{k\neq\pm n}\left\vert
(\Psi_{n,j,t},e^{i(2\pi k+t)x})\right\vert ^{2}=O\left(  \frac{1}{n^{2}%
}\right)  .
\]
Therefore decomposing $\Psi_{n,j,t}$ by basis $\{e^{i(2\pi k+t)x}%
:k\in\mathbb{Z}\}$ we get (26) and (27). The normalization condition
$\left\Vert \Psi_{n,j,t}\right\Vert =1$ with (26) and (27) implies (28)
\end{proof}

Using (24) in (23), replacing $k$ and $m$ by $0$ and $n_{1}$ respectively and
then isolating the term containing the multiplicand $(\Psi_{n,j,t},e^{i(-2\pi
n+t)x})$ we obtain%
\begin{equation}
(\lambda_{n,j}(t)-(2\pi n+t)^{2})(\Psi_{n,j,t},e^{i(2\pi n+t)x})-q_{2n}%
(\Psi_{n,j,t},e^{i(-2\pi n+t)x})=
\end{equation}%
\[
\sum_{n_{1}\neq0,2n;\text{ }n_{1}=-\infty}^{\infty}q_{n_{1}}(\Psi
_{n,j,t},e^{i(2\pi(n-n_{1})+t)x}).
\]
Now we iterate (30) by using the formula
\begin{equation}
(\Psi_{n,j,t},e^{i(2\pi(n-n_{1})+t)x})=\sum_{n_{2}=-\infty}^{\infty}%
\frac{q_{n_{2}}(\Psi_{n,j,t},e^{i(2\pi(n-n_{1}-n_{2})+t)x})}{\lambda
_{n,j}(t)-(2\pi(n-n_{1})+t)^{2}}%
\end{equation}
obtained from (23) and (24). Taking into account that the denominator of the
fraction in (31) is a large number for $n_{1}\neq0,2n$ and $t\in\lbrack
0,\rho]$ (see (22)), we iterate (30) as follows. Since this iteration is
similar to that done in [1], here we give only the scheme of this iteration.
First, we use (31) in (30), replacing the terms $(\Psi_{n,j,t},e^{i(2\pi
(n-n_{1})+t)x})$ for $n_{1}\neq0,2n$ in (30) by the right-hand side of (31)
and get the summation with respect to $n_{1}$ and $n_{2}$ in the right-hand
side of (30). We then isolate in this summation the terms containing one of
the multiplicands $(\Psi_{n,j,t},e^{i(2\pi n+t)x})$ , $(\Psi_{n,j,t}%
,e^{i(-2\pi n+t)x})$ (i.e., terms with $n_{1}+n_{2}=0,2n$ ) and use (31) in
the other terms. Repeating this process $m-$times, we obtain
\begin{equation}
(\lambda_{n,j}(t)-(2\pi n+t)^{2}-A_{m}(\lambda_{n,j}(t),t))u_{n,j}%
(t)=(q_{2n}+B_{m}(\lambda_{n,j}(t),t))v_{n,j}(t)+R_{m},
\end{equation}
where
\begin{equation}
A_{m}(\lambda_{n,j}(t),t)=\sum_{k=1}^{m}a_{k}(\lambda_{n,j}(t),t),\text{
}B_{m}(\lambda_{n,j}(t),t)=\sum_{k=1}^{m}b_{k}(\lambda_{n,j}(t),t),
\end{equation}%
\[
a_{k}(\lambda_{n,j}(t),t)=\sum_{n_{1},n_{2},...,n_{k}}\frac{q_{n_{1}}q_{n_{2}%
}...q_{n_{k}}q_{-n_{1}-n_{2}-...-n_{k}}}{[\lambda_{n,j}-(2\pi(n-n_{1}%
)+t)^{2}]...[\lambda_{n,j}-(2\pi(n-n_{1}-...-n_{k})+t)^{2}]},
\]%
\[
b_{k}(\lambda_{n,j}(t),t)=\sum_{n_{1},n_{2},...,n_{k}}\frac{q_{n_{1}}q_{n_{2}%
}...q_{n_{k}}q_{2n-n_{1}-n_{2}-...-n_{k}}}{[\lambda_{n,j}-(2\pi(n-n_{1}%
)+t)^{2}]...[\lambda_{n,j}-(2\pi(n-n_{1}-..-n_{k})+t)^{2}]},
\]%
\[
R_{m}=\sum_{n_{1},n_{2},...,n_{m+1}}\frac{q_{n_{1}}q_{n_{2}}...q_{n_{m}%
}q_{n_{m+1}}(q\Psi_{n,j,t},e^{i(2\pi(n-n_{1}-...-n_{m+1})+t)x})}%
{[\lambda_{n,j}-(2\pi(n-n_{1})+t)^{2}]...[\lambda_{n,j}-(2\pi(n-n_{1}%
-...-n_{m+1})+t)^{2}]}.
\]
Note that, here the sums are taken under conditions $n_{s}\neq0$ and
$n_{1}+n_{2}+...+n_{s}\neq0,2n$ for $s=1,2,....$ Using (22), (25) and (29) one
can easily verify that the equalities
\begin{equation}
a_{k}=O\left(  (\frac{\ln\left\vert n\right\vert }{n})^{k}\right)  ,\text{
}b_{k}=O\left(  (\frac{\ln\left\vert n\right\vert }{n})^{k}\right)  ,\text{
}R_{m}=O\left(  (\frac{\ln\left\vert n\right\vert }{n})^{m+1}\right)
\end{equation}
hold uniformly with respect to $t$ in $[0,\rho].$

In the same way the relation
\begin{equation}
(\lambda_{n,j}(t)-(-2\pi n+t)^{2}-A_{m}^{^{\prime}}(\lambda_{n,j}%
(t),t))v_{n,j}(t)=(q_{-2n}+B_{m}^{^{\prime}}(\lambda_{n,j}(t),t))u_{n,j}%
(t)=R_{m}^{^{\prime}}%
\end{equation}
can be obtained, where
\[
A_{m}^{^{\prime}}(\lambda_{n,j}(t),t)=\sum_{k=1}^{m}a_{k}^{\prime}%
(\lambda_{n,j}(t),t),\text{ }B_{m}^{^{\prime}}(\lambda_{n,j}(t))=\sum
_{k=1}^{m}b_{k}^{^{\prime}}(\lambda_{n,j}(t),t),
\]%
\[
a_{k}^{^{\prime}}(\lambda_{n,j}(t),t)=\sum_{n_{1},n_{2},...,n_{k}}%
\frac{q_{n_{1}}q_{n_{2}}...q_{n_{k}}q_{-n_{1}-n_{2}-...-n_{k}}}{[\lambda
_{n,j}-(2\pi(n+n_{1})-t)^{2}]...[\lambda_{n,j}-(2\pi(n+n_{1}+...+n_{k}%
)-t)^{2}]},
\]

\[
b_{k}^{^{\prime}}(\lambda_{n,j}(t),t)=\sum_{n_{1},n_{2},...,n_{k}}%
\frac{q_{n_{1}}q_{n_{2}}...q_{n_{k}}q_{-2n-n_{1}-n_{2}-...-n_{k}}}%
{[\lambda_{n,j}-(2\pi(n+n_{1})-t)^{2}]...[\lambda_{n,j}-(2\pi(n+n_{1}%
+...+n_{k}-t))^{2}]},
\]

\begin{equation}
a_{k}^{^{\prime}}=O\left(  (\frac{ln\left\vert n\right\vert }{n})^{k}\right)
,\text{ }b_{k}^{^{\prime}}=O\left(  (\frac{ln\left\vert n\right\vert }{n}%
)^{k}\right)  ,\text{ }R_{m}^{^{\prime}}=O\left(  (\frac{ln\left\vert
n\right\vert }{n})^{m+1}\right)  ,
\end{equation}
$n_{s}\neq0,$ $n_{1}+n_{2}+...+n_{s}\neq0,-2n$ for $s=1,2,...,k.$

Now in (32) and (35) letting $\ m$ tend to infinity, using (33), (34) and (36)
we obtain
\begin{equation}
(\lambda_{n,j}(t)-(2\pi n+t)^{2}-A(\lambda_{n,j}(t),t))u_{n,j}(t)=(q_{2n}%
+B(\lambda_{n,j}(t),t))v_{n,j}(t),
\end{equation}
\begin{equation}
(\lambda_{n,j}(t)-(-2\pi n+t)^{2}-A^{^{\prime}}(\lambda_{n,j}(t),t))v_{n,j}%
(t)=(q_{-2n}+B^{^{\prime}}(\lambda_{n,j}(t),t))u_{n,j}(t),
\end{equation}
where
\begin{equation}
A(\lambda,t)=\sum_{k=1}^{\infty}a_{k}(\lambda,t),\text{ }B(\lambda
,t)=\sum_{k=1}^{\infty}b_{k}(\lambda,t),\text{ }A^{^{\prime}}=\sum
_{k=1}^{\infty}a_{k}^{\prime},\text{ }B^{^{\prime}}=\sum_{k=1}^{\infty}%
b_{k}^{^{\prime}}.
\end{equation}

The simplicity of the eigenvalues $\lambda_{n,j}(t)$ for $t\in\lbrack0,\rho]$,
$n>N$, under some conditions on the potential $q,$ are obtained from formulas
(37) and (38) by the following scheme: First, we estimate the functions
$A(\lambda,t),B(\lambda,t),$ $A^{^{\prime}}(\lambda,t)$ and $B^{^{\prime}%
}(\lambda,t)$ (see Lemma 2). Next, using \ Lemma 2 we prove that the
eigenvalues $\lambda_{n,j}(t)$ of $L_{t}(q)$ for $t\in\lbrack0,\rho]$ and
$n>N$ are the roots of the equation
\begin{equation}
(\lambda-(2\pi n+t)^{2}-A(\lambda,t))(\lambda-(2\pi n-t)^{2}-A^{^{\prime}%
}(\lambda,t))=(q_{2n}+B(\lambda,t))(q_{-2n}+B^{^{\prime}}(\lambda,t))
\end{equation}
if $q$ satisfies some conditions (see Theorem 1). Then $\lambda_{n,j}(t)$
satisfies at least one of the equations
\begin{equation}
\lambda=(2\pi n+t)^{2}+\frac{1}{2}(A(\lambda,t)+A^{^{\prime}}(\lambda,t))-4\pi
nt-\sqrt{D(\lambda,t)}%
\end{equation}
and
\begin{equation}
\lambda=(2\pi n+t)^{2}+\frac{1}{2}(A(\lambda,t)+A^{^{\prime}}(\lambda,t))-4\pi
nt+\sqrt{D(\lambda,t)},
\end{equation}
where%
\begin{equation}
D(\lambda,t)=(4\pi nt)^{2}+q_{2n}q_{-2n}+D_{1}(\lambda,t)+D_{2}(\lambda,t),
\end{equation}%
\begin{align}
D_{1}(\lambda,t)  &  =8\pi ntC(\lambda,t)+C^{2}(\lambda,t),\text{ }%
C(\lambda,t)=\frac{1}{2}(A(\lambda,t)-A^{^{\prime}}(\lambda,t)),\\
D_{2}(\lambda,t)  &  =q_{2n}B^{^{\prime}}(\lambda,t)+q_{-2n}B(\lambda
,t)+B(\lambda,t)B^{^{\prime}}(\lambda,t)\text{.}\nonumber
\end{align}
To prove\ the simplicity of the eigenvalues $\lambda_{n,1}(t)$ and
$\lambda_{n,2}(t)$ for $t\in\lbrack0,\rho]$ and $n>N,$ we show that one of
these eigenvalues satisfies (41) and the other one satisfies (42) and the
roots of (41) and (42) are different. For this, we prove that the functions
$B(\lambda,t),$ $B^{^{\prime}}(\lambda,t),$ $C(\lambda,t)$ (Lemma 3) and
$\sqrt{D(\lambda,t)}$ (Lemma 4) satisfy some Lipschitz conditions. As a
result, we find the conditions on $q$ that guarantee the simplicity of those
eigenvalues (Theorem 2).

\begin{lemma}
$(a)$ The following equalities hold uniformly with respect to $t$ in
$[0,\rho]:$%
\begin{equation}
A(\lambda_{n,j}(t),t)=O(n^{-1}),\ A^{^{\prime}}(\lambda_{n,j}(t),t)=O(n^{-1}).
\end{equation}

$(b)$ \textit{Let }$q\in W_{1}^{p}[0,1]$\textit{ and (3) holds with some
}$s\leq p.$ Then the equalities
\begin{equation}
B(\lambda_{n,j}(t),t)=o\left(  n^{-s-1}\right)  ,\text{ }B^{^{\prime}}%
(\lambda_{n,j}(t),t)=o\left(  n^{-s-1}\right)
\end{equation}
hold uniformly with respect to $t$ in $[0,\rho].$
\end{lemma}

\begin{proof}
$(a)$ First, let us prove that
\begin{equation}
a_{1}(\lambda_{n,j}(t),t)=\frac{1}{4\pi^{2}}\sum_{k\neq0,\,2n}\frac
{q_{k}\,q_{-k}}{k(2n-k)}+O\left(  \frac{1}{n}\right)  .
\end{equation}
Using (20) and the inequality $t<\rho$ and taking into account that $\rho$ is
a sufficiently small fixed number, one can see that if $\left\vert
k\right\vert \leq3\left\vert n\right\vert ,$ then
\[
\mid\lambda_{n,j}-(2\pi(n-k)+t)^{2}-4\pi^{2}k(2n-k)\mid\leq\left\vert
n\right\vert .
\]
Conversely, if $\left\vert k\right\vert >3\left\vert n\right\vert $ then
\[
\mid\lambda_{n,j}-(2\pi(n-k)+t)^{2}\mid>k^{2}>n^{2},\text{ }\mid4\pi
^{2}k(2n-k)\mid>k^{2}>n^{2}.
\]
Therefore, taking into account the inequality in (29), we obtain%
\[
\sum_{k:\left\vert k\right\vert >3\left\vert n\right\vert }\frac{q_{k}%
\,q_{-k}}{\lambda_{n,j}-(2\pi(n-k)+t)^{2}}-\frac{1}{4\pi^{2}}\sum
_{k:\left\vert k\right\vert >3\left\vert n\right\vert }\frac{q_{k}\,q_{-k}%
}{k(2n-k)}=O\left(  \frac{1}{n}\right)  ,
\]%
\[
\mid\sum_{k:\left\vert k\right\vert \leq3\left\vert n\right\vert ,k\neq
0,\,2n}(\frac{q_{k}\,q_{-k}}{\lambda_{n,j}-(2\pi(n-k)+t)^{2}}-\frac
{q_{k}\,q_{-k}}{4\pi^{2}k(2n-k)})\mid\leq
\]%
\[
\sum_{k:\left\vert k\right\vert \leq3\left\vert n\right\vert }\frac
{M^{2}\left\vert n\right\vert }{k^{2}(2n-k)^{2}}\leq\sum_{k:\left\vert
k\right\vert \leq\left\vert n\right\vert }\frac{M^{2}\left\vert n\right\vert
}{k^{2}n^{2}}+\sum_{k:\left\vert n\right\vert <\left\vert k\right\vert
\leq3\left\vert n\right\vert }\frac{M^{2}\left\vert n\right\vert }%
{n^{2}(2n-k)^{2}}=O(\frac{1}{n}).
\]
Thus (47) holds. In (47) grouping the terms $\frac{q_{k}\,q_{-k}}{k(2n-k)}$
and $\frac{q_{-k}\,q_{k}}{-k(2n+k)},$ we get
\begin{equation}
a_{1}(\lambda_{n,j}(t),t)=\frac{1}{4\pi^{2}}\sum_{k>0,\text{ }k\neq\,2n}%
\frac{q_{k}\,q_{-k}}{(2n+k)(2n-k)}+O\left(  \frac{1}{n}\right)  .
\end{equation}
To estimate the sum in (48) we consider, as done in [22], the function
\[
G(x,n)=\int_{0}^{x}q(t)e^{-2\pi i(2n)t}dt-q_{2n}x.
\]
The Fourier coefficients $G_{k}(n)=:(G(x,n),e^{2\pi ikx})$ of $G(x,n)$ are
\[
G_{k}({n})=\frac{1}{2\pi ik}q_{2n+k}%
\]
for$\ k\neq0,$ and hence we have
\[
G(x,n)=G_{0}(n)+\sum_{k\neq2n}\frac{q_{k}}{2\pi i(k-2n)}e^{2\pi i(k-2n)x}.
\]
Therefore, using the integration by parts and taking into account the obvious
equalities $G(1,n)=G(0,n)=0,$ $G(x,n)-G_{0}(n)=O(1),$ we obtain
\[
\frac{1}{4\pi^{2}}\sum_{k>0,\text{ }k\neq\,2n}\frac{q_{k}\,q_{-k}%
}{(2n+k)(2n-k)}=\int_{0}^{1}(G(x,n)-G_{0}(n))^{2}e^{2\pi i(4n)x}\,dx=
\]%
\[
\frac{-1}{2\pi i(4n)}\int_{0}^{1}2(G(x,n)-G_{0}(n))(q(x)e^{-2\pi
i(2n)x}-q_{2n})e^{2\pi i(4n)x}\,dx=O\left(  \frac{1}{n}\right)  .
\]
This with (48) and (34) implies the first equality of (45). In the same way,
we get the second equality of (45).

$(b)$ If the assumptions of $(b)$ hold, then
\begin{equation}
q_{2n}=o(n^{-s}),\text{ }q_{n_{1}}q_{n_{2}}\cdots q_{n_{k}}q_{\pm
2n-n_{1}-n_{2}-\dots-n_{k}}=o(n^{-s})
\end{equation}
(see p. 655 of [22]). Using \ this and (22), in a standard way, we get
\begin{equation}
b_{k}(\lambda_{n,j}(t))=o\left(  \frac{\ln^{k}\,n}{n^{k+s}}\right)
=o(n^{-s-1}),\quad b_{k}^{^{\prime}}(\lambda_{n,j}(t))=o\left(  \frac{\ln
^{k}\,n}{n^{k+s}}\right)  =o(n^{-s-1})
\end{equation}
for$\ k\geq2$. Now, it remains to prove that
\begin{equation}
b_{1}(\lambda_{n,j}(t))=o(n^{-s-1}),\quad b_{1}^{^{\prime}}(\lambda
_{n,j}(t))=o(n^{-s-1}).
\end{equation}
Instead of the inequality in (29) using the equality $q_{k}q_{2n-k}=o(n^{-s})$
(see (49)) and arguing as in the proof of (47) we get \
\begin{equation}
b_{1}(\lambda_{n,j}(t),t)=\frac{1}{4\pi^{2}}\sum_{k\neq0,\,2n}\frac
{q_{k}\,q_{2n-k}}{k(2n-k)}+o(n^{-s-1})
\end{equation}
for all $t$ $\in\lbrack0,\rho].$ In [22] (see p. 655) the summation in (52) is
denoted by $S_{2n}$ and it is proved that $S_{2n}=o(n^{-s-1})$ (see p. 658 ).
Thus, from (52) we obtain the first equality of (51). In the same way, we get
the second equality of (51). Now, (46) follows from (50) and (51)
\end{proof}

\begin{theorem}
\textit{ Let }$q\in W_{1}^{p}[0,1]$\textit{ and (3) holds with some }$s\leq
p.$ Suppose (5) holds and
\begin{equation}
\mid q_{2n}\mid>cn^{-s-1}%
\end{equation}
for some $c>0.$ Then the eigenvalues $\lambda_{n,j}(t)$ of $L_{t}(q)$ for
$t\in\lbrack0,\rho]$ and $n>N$ are the roots of equation (40).
\end{theorem}

\begin{proof}
It follows from (53) and (46) that
\begin{equation}
q_{2n}+B(\lambda_{n,j}(t),t)\neq0,\text{ }q_{-2n}+B^{^{\prime}}(\lambda
_{n,j}(t),t)\neq0
\end{equation}
for $t\in\lbrack0,\rho].$ Let us prove that this with formulas (37), (38) and
(28) gives
\begin{equation}
u_{n,j}(t)v_{n,j}(t)\neq0.
\end{equation}
If $u_{n,j}(t)=0$ then by (28) $v_{n,j}(t)\neq0$ and by (37) $q_{2n}%
+B(\lambda_{n,j}(t),t)=0$ which contradicts (54). Similarly, if $v_{n,j}(t)=0$
then by (28) and (38) $q_{-2n}+B^{^{\prime}}(\lambda_{n,j}(t),t)=0$ which
again contradicts (54). Therefore multiplying (37) and (38) side by side and
then using (55), we get the proof of the theorem.
\end{proof}

Now we consider some properties of the functions $\ A(\lambda,t),$
$B(\lambda,t),$ $A^{^{\prime}}(\lambda,t),$ $B^{^{\prime}}(\lambda,t)$ defined
in (39) for $t\in\lbrack0,\rho]$ and $\lambda\in D(n,t,\rho),$ where
$D(n,t,\rho)$ is the disk defined in (21).

\begin{lemma}
$(a)$ There exists a constant $K$, independent of $\ n>N$ and $t\in
\lbrack0,\rho],$ such that%
\begin{equation}
\mid A(\lambda,t)-A(\mu,t)\mid<Kn^{-2}\mid\lambda-\mu\mid,\text{ }\mid
A^{^{\prime}}(\lambda,t)-A^{^{\prime}}(\mu,t)\mid<Kn^{-2}\mid\lambda-\mu\mid,
\end{equation}%
\begin{equation}
\mid C(\lambda_{n,j}(t),t)\mid<tKn^{-1},\text{ }\mid C(\lambda,t))-C(\mu
,t))\mid<tKn^{-2}\mid\lambda-\mu\mid
\end{equation}
for all $\lambda,\mu\in D(n,t,\rho),$ where $N$ is defined in Remark 1.

$(b)$ \textit{Let }$q\in W_{1}^{p}[0,1],$ \textit{and (3) holds with some
}$s\leq p.$ Then the functions $b_{k}(\lambda,t),$ $b_{k}^{^{\prime}}%
(\lambda,t),$ $B(\lambda,t),$ $B^{^{\prime}}(\lambda,t)$ for $\lambda,\mu\in
D(n,t,\rho),$ $k=1,2,...,$ satisfy the following, uniform with respect to $t$
in $[0,\rho],$ equality
\begin{equation}
f(\lambda,t)-f(\mu,t)=(\lambda-\mu)o(n^{-s-2}).
\end{equation}

\end{lemma}

\begin{proof}
$(a)$ If $\lambda\in D(n,t,\rho),$ then%
\begin{equation}
\left\vert \lambda-(2\pi(n-k)+t)^{2}\right\vert >\left\vert k\right\vert
\left\vert 2n-k\right\vert ,\text{ }\forall k\neq0,2n.
\end{equation}
To prove estimations (56) and (57) we use (59) and the following obvious
equality
\begin{equation}
\sum_{k\neq0,-2n}\frac{1}{\mid k^{s}(2n-k)^{m}\mid}=O\left(  \frac{1}{n^{p}%
}\right)
\end{equation}
if $\max\{s,m\}\geq2,$ where $p=\min\{s,m\}\geq1.$ By (60) and (59) the series
in the formulas for the functions $a_{k}(\lambda,t),$ $a_{k}^{^{\prime}%
}(\lambda,t),$ $b_{k}(\lambda,t),$ $b_{k}^{^{\prime}}(\lambda,t)$ converge
uniformly in a neighborhood of $\lambda.$ It implies that these functions
continuously depend on $\lambda.$ Moreover, estimations (34) and (36) hold if
we replace $\lambda_{n,j}(t)$ by $\lambda.$ Therefore the series in the
formulas for the functions $A(\lambda,t),$ $B(\lambda,t),$ $A^{^{\prime}%
}(\lambda,t)$ and $B^{^{\prime}}(\lambda,t)$ converge uniformly in a
neighborhood of $\lambda.$ Using (59) and (60) one can easily verify that
these series can be differentiated, with respect to $\lambda,$ term by term.
Moreover, taking into account the inequality
\[
\mid\frac{d}{d\lambda}(\frac{1}{\lambda-(2\pi(n-k)+t)^{2}})\mid\leq\frac
{1}{k^{2}(2n-k)^{2}}%
\]
and (60), we see that the absolute values of the derivatives of $a_{k}%
(\lambda,t),$ $a_{k}^{^{\prime}}(\lambda,t),$ $b_{k}(\lambda,t),$
$b_{k}^{^{\prime}}(\lambda,t)$ \ with respect to $\lambda$ is $O(n^{-k-1}).$
Therefore, these functions satisfy the condition
\begin{equation}
g(\lambda,t)-g(\mu,t)=(\lambda-\mu)O(n^{-k-1}).
\end{equation}
Now (56) follows from (61).

To prove the first inequality of (57) we use substitutions $-n_{1}-n_{2}%
-\dots-n_{k}=j_{1},$

$\ \,n_{2}=j_{k},\ \,n_{3}=j_{k-1},\,\dots,\,n_{k}=j_{2}$ in the formula for
the expression $a_{k}^{\prime}.$ Then the inequalities for the forbidden
indices $n_{p}\neq0,$ $n_{1}+n_{2}+\dots+n_{p}\neq0,-2n$ for $1\leq p\leq k$
in the formula for $a_{k}^{\prime}$ take the form $j_{p}\neq0,$ $j_{1}%
+j_{2}+...+j_{p}\neq0,2n$ for $1\leq p\leq k$, and
\[
a_{k}^{^{\prime}}(\lambda_{n,j}(t))=\sum_{n_{1},n_{2},...,n_{k}}\frac
{q_{n_{1}}q_{n_{2}}...q_{n_{k}}q_{-n_{1}-n_{2}-...-n_{k}}}{[\lambda
_{n,j}-(2\pi(n-n_{1})-t)^{2}]...[\lambda_{n,j}-(2\pi(n-n_{1}-...-n_{k}%
)-t)^{2}]}.
\]
Using (22) and (60) one can readily see that
\[
\sum_{\substack{k=-\infty,\\k\neq0,2n}}^{\infty}\mid\frac{1}{\lambda
_{n,j}(t)-(2\pi(n-k)+t)^{2}}-\frac{1}{\lambda_{n,j}(t)-(2\pi(n-k)-t)^{2}}%
\mid=tO(\frac{1}{n}).
\]
This with the inequality in (29) implies the first inequality in (57). Now
arguing as in the proof of (56), we get the proof of \ the second inequality
of (57).

$(b)$ Using (49) and repeating the proof of (56) we get the proof of $(b)$
\end{proof}

\begin{lemma}
Suppose the conditions of Theorem 1 hold. If at least one of the inequalities
\begin{equation}
\operatorname{Re}q_{2n}q_{-2n}\geq0,
\end{equation}%
\begin{equation}
\mid\operatorname{Im}q_{2n}q_{-2n}\mid\geq\varepsilon\mid q_{2n}q_{-2n}\mid
\end{equation}
are satisfied for some $\varepsilon>0$ and for $n>N,$ where $N$ is defined in
Remark 1, then there exists a constant $c_{1}$ such that
\begin{equation}
\left\vert \sqrt{D(\lambda_{n,1}(t),t)}-\sqrt{D(\lambda_{n,2}(t),t)}%
\right\vert <c_{1}n^{-1}\mid\lambda_{n,1}(t)-\lambda_{n,2}(t)\mid
\end{equation}
for $\ n>N$ and $t\in\lbrack0,\rho].$
\end{lemma}

\begin{proof}
First let us prove that%
\begin{equation}
\left\vert D(\lambda_{n,j}(t),t)\right\vert >\frac{\varepsilon}{4}(\left\vert
q_{-2n}q_{2n}\right\vert +(4\pi nt)^{2}),
\end{equation}%
\begin{equation}
D(\lambda_{n,j}(t),t)=((4\pi nt)^{2}+q_{2n}q_{-2n})(1+o(1))
\end{equation}
for $\ n>N$ and $t\in\lbrack0,\rho].$ If follows from (44), (57) and (53),
(46), (5) that
\begin{equation}
D_{1}(\lambda_{n,j}(t),t)=t^{2}O(1),\text{ }D_{2}(\lambda_{n,j}(t),t)=o(q_{2n}%
n^{-s-1})=o(q_{2n}q_{-2n}).
\end{equation}
Therefore, we have%
\begin{equation}
D_{1}(\lambda_{n,j}(t),t)+\text{ }D_{2}(\lambda_{n,j}(t),t)=o(\left\vert
q_{-2n}q_{2n}\right\vert +(4\pi nt)^{2}).
\end{equation}
Thus, by (43), to prove (65) and (66) it is enough to show that%
\begin{equation}
\left\vert q_{-2n}q_{2n}+(4\pi nt)^{2}\right\vert >\frac{\varepsilon}%
{3}(\left\vert q_{-2n}q_{2n}\right\vert +(4\pi nt)^{2}).
\end{equation}
\ For this we consider \ two cases. First case: $(4\pi nt)^{2}\leq2\left\vert
q_{2n}q_{-2n}\right\vert ).$ Then
\begin{equation}
\left\vert q_{2n}q_{-2n}\right\vert \geq\frac{1}{3}(\left\vert q_{2n}%
q_{-2n}\right\vert +(4\pi nt)^{2}).
\end{equation}
If (62) holds then$\ \left\vert q_{2n}q_{-2n}+(4\pi nt)^{2}\right\vert
\geq\left\vert q_{2n}q_{-2n}\right\vert .$ Therefore, (69) follows from (70).
If (63) holds then$\ \left\vert q_{2n}q_{-2n}+(4\pi nt)^{2}\right\vert
\geq\mid\operatorname{Im}q_{2n}q_{-2n}\mid\geq\varepsilon\left\vert
q_{2n}q_{-2n}\right\vert $ and again (69) follows from (70). Now let us
consider the second case: $(4\pi nt)^{2}>2\left\vert q_{2n}q_{-2n}\right\vert
).$ Then%
\[
\left\vert q_{2n}q_{-2n}+(4\pi nt)^{2}\right\vert >(4\pi nt)^{2}-\left\vert
q_{2n}q_{-2n}\right\vert >\frac{1}{3}(\left\vert q_{2n}q_{-2n}\right\vert
+(4\pi nt)^{2}),
\]
that is, (69) holds. Thus, (65) and (66) are proved.

Using (43), (57), (49), (46) and Lemma 3$(b)$ one can easily verify that%
\begin{equation}
\mid D(\lambda_{n,1}(t),t)-D(\lambda_{n,2}(t),t)\mid\leq(5\pi t^{2}%
Kn^{-1}+n^{-2s-2})\mid\lambda_{n,1}(t)-\lambda_{n,2}(t)\mid.
\end{equation}
On the other hand, it follows from (67), (66) and (53), (5) that there exists
a constant $c_{2}$ such that
\begin{equation}
\left\vert \sqrt{D(\lambda_{n,1}(t),t)}+\sqrt{D(\lambda_{n,2}(t),t)}%
\right\vert =\left\vert 2\sqrt{D(\lambda_{n,1}(t),t)}(1+o(1))\right\vert
>c_{2}(n^{-s-1}+nt).
\end{equation}
Thus, from (71) and (72) we obtain (64).
\end{proof}

Now using Lemma 2-4 and Theorem 1 we prove the following main result.

\begin{theorem}
\textit{ Let }$q\in W_{1}^{p}[0,1]$\textit{ and (3) holds with some }$s\leq
p.$ Suppose (5) and (53) are satisfied. \textit{If at least one of the
inequalities (62) and (63) hold,} then the eigenvalues $\lambda_{n,j}(t)$ of
$L_{t}(q)$ for $n>N,$ $j=1,2$ and $t\in\lbrack0,\rho]$ are simple and
\begin{equation}
\lambda_{n,j}(t)=(2\pi n+t)^{2}+\frac{1}{2}(A(\lambda_{n,j}(t),t)+A^{^{\prime
}}(\lambda_{n,j}(t),t))-4\pi nt+(-1)^{j}\sqrt{D(\lambda_{n,j}(t),t)}.
\end{equation}

\end{theorem}

\begin{proof}
If both eigenvalues $\lambda_{n,1}(t)$ and $\lambda_{n,2}(t)$ of the operator
$L_{t}(q)$ lying in the disk $D(n,t,\rho)$ (see (21)) satisfy equation (41),
then
\begin{equation}
\lambda_{n,1}(t)-\lambda_{n,2}(t)=[\frac{1}{2}(A(\lambda_{n,1}(t),t)-A(\lambda
_{n,2}(t),t))+\frac{1}{2}(A^{^{\prime}}(\lambda_{n,1}(t),t)-A^{^{\prime}%
}(\lambda_{n,2}(t),t)]+
\end{equation}%
\[
\left[  \sqrt{D(\lambda_{n,1}(t),t)}-\sqrt{D(\lambda_{n,2}(t),t)}\right]  .
\]
By (56) we have%
\begin{equation}
\mid A(\lambda_{n,1},t)-A(\lambda_{n,2},t)+A^{^{\prime}}(\lambda
_{n,1},t)-A^{^{\prime}}(\lambda_{n,2},t)\mid<2Kn^{-2}\mid\lambda
_{n,1}(t)-\lambda_{n,2}(t)\mid.
\end{equation}
Thus, using this and (64) in (74) we get%
\begin{equation}
\lambda_{n,1}(t)=\lambda_{n,2}(t).
\end{equation}
In the same way, we prove that if both $\lambda_{n,1}(t)$ and $\lambda
_{n,2}(t)$ satisfy (42) then (76) holds.

Now suppose that one of them, say $\lambda_{n,1}(t),$ satisfies (41) and the
other $\lambda_{n,2}(t)$ satisfies (42). Then
\begin{equation}
\lambda_{n,1}(t)-\lambda_{n,2}(t)=[\frac{1}{2}(A(\lambda_{n,1},t)-A(\lambda
_{n,2},t))+\frac{1}{2}(A^{^{\prime}}(\lambda_{n,1},t)-A^{^{\prime}}%
(\lambda_{n,2},t)]+
\end{equation}%
\[
\left[  \sqrt{D(\lambda_{n,1}(t),t)}+\sqrt{D(\lambda_{n,2}(t),t)}\right]  .
\]
Therefore, by (75) and (72) there exists a constant $c_{3}$ such that
\begin{equation}
\mid\lambda_{n,1}(t)-\lambda_{n,2}(t)\mid>c_{3}(n^{-s-1}+nt).
\end{equation}

Now it follows from (76) and (78) that the value $d_{n}(t)$ of the function
$d_{n}$, defined in (19), for $t\in\lbrack0,\rho]$ belongs to the union of the
disjoint sets $(c_{3}(n^{-s-1}+nt),\infty)$ and $\{0\}.$ Moreover, as it is
proved in Remark 1, $d_{n}$ is a continuous function on $[0,\rho]$ which
implies that the set $\{d_{n}(t):t\in\lbrack0,\rho]\}$ is a connected set.
Therefore, taking into account that $d_{n}(\rho)\in(c_{3}(n^{-s-1}%
+nt),\infty)$ (see (9)), we get $\{d_{n}(t):t\in\lbrack0,\rho]\}\in
(c_{3}(n^{-s-1}+nt),\infty).$ This means that $\lambda_{n,1}(t)$ and
$\lambda_{n,2}(t)$ are different simple eigenvalues and one of them satisfies
(41) and the other satisfies (42). Without loss of generality, it can be
assumed that (73) holds where square root in (73) is taken with positive real
part. Note that
\begin{equation}
\operatorname{Re}(\sqrt{D(\lambda_{n,j}(t),t)})\neq0
\end{equation}
due to the following reason. By (53), (62)) and (63) $\left\vert \arg
(q_{-2n}q_{2n}+(4\pi nt)^{2})\right\vert $ $<\pi-\alpha,$ where $\alpha
\in(0,\pi).$ Thus, by (65) and (66) $\arg D(\lambda_{n,j}(t),t)\neq\pi$ and
hence (79) holds
\end{proof}

Now to prove the uniform asymptotic formulas for the eigenfunctions
$\Psi_{n,j,t}(x)$ we need consider $v_{n,j}(t)$ and $u_{n,j}(t)$ (see
(26)-(28)). Namely, we use the following

\begin{lemma}
Suppose that all conditions of Theorem 2 hold. Let $\lambda_{n,j}(t)$ be the
eigenvalue of $L_{t}(q)$ satisfying (73), and $\Psi_{n,j,t}(x)$ be the
corresponding eigenfunction. Then the relations
\begin{equation}
v_{n,1}(t)\sim1,\text{ }u_{n,2}(t)\sim1
\end{equation}
hold uniformly for $t\in\lbrack0,\rho].$
\end{lemma}

\begin{proof}
Multiplying (37) and (38 ) by $v_{n,j}(t)$ and by $u_{n,j}(t)$ respectively
and then subtracting each other, we get
\begin{equation}
(-8\pi nt+A^{^{\prime}}(t)-A(t))u_{n,j}(t)v_{n,j}(t)=(q_{2n}+B(t))v_{n,j}%
^{2}(t)-(q_{-2n}+B^{^{\prime}}(t))u_{n,j}^{2}(t),
\end{equation}
where, for brevity, $A(\lambda_{n,j}(t),t),A^{^{\prime}}(\lambda
_{n,j}(t),t),B(\lambda_{n,j}(t),t)$ and $B^{^{\prime}}(\lambda_{n,j}(t),t)$
are denoted by $A(t),A^{^{\prime}}(t),B(t)$ and $B^{^{\prime}}(t)$ respectively.

First, suppose that $nt\leq\left\vert q_{2n}\right\vert .$ Then it follows
from (57) that $A^{^{\prime}}(t)-A(t)=o(q_{2n})$ and%
\begin{equation}
\mid-8\pi nt+A^{^{\prime}}(t)-A(t)\mid<9\pi\left\vert q_{2n}\right\vert .
\end{equation}
On the other hand, relations (5), (46) and (53) imply that
\begin{equation}
q_{2n}+B(\lambda_{n,j}(t),t)\sim q_{-2n}+B^{^{\prime}}(\lambda_{n,j}(t),t)\sim
q_{2n}.
\end{equation}
Therefore, using (81)-(83) and taking into account that if the relation
$u_{n,j}(t)\sim v_{n,j}(t)$ does not hold then $u_{n,j}(t)v_{n,j}(t)=o(1)$
(see (28)), we obtain $u_{n,j}(t)\sim v_{n,j}(t)\sim1$ \ for $j=1,2$. Thus,
(80) holds for the case $nt\leq\left\vert q_{2n}\right\vert .$

Now consider the case $nt>\left\vert q_{2n}\right\vert .$ Using (73) in (37)
and (38), we obtain
\begin{equation}
(C(t)-4\pi nt+(-1)^{j}\sqrt{D(t)})u_{n,j}(t)=(q_{2n}+B(t))v_{n,j}(t),
\end{equation}%
\begin{equation}
(-C(t)+4\pi nt+(-1)^{j}\sqrt{D(t)})v_{n,j}(t)=(q_{-2n}+B^{^{\prime}%
}(t))u_{n,j}(t).
\end{equation}
Since $\operatorname{Re}(\sqrt{D(\lambda_{n,j}(t),t)})>0,$ it follows from
(84) for $j=1$ and (57) that
\[
\mid C(\lambda_{n,1}(t),t)-4\pi nt-\sqrt{D(\lambda_{n,1}(t),t)}\mid
\geq\operatorname{Re}(4\pi nt(1+O(n^{-2}))+\sqrt{D(\lambda_{n,1}%
(t),t)})>\left\vert q_{2n}\right\vert .
\]
Using this and (83) in (84) for $j=1$ we get $v_{n,1}(t)\sim1.$ In the same
way we obtain the second relation of (80) from (85) for $j=2.$
\end{proof}

To obtain the asymptotic formulas of arbitrary accuracy we define successively
the following functions%
\begin{align*}
F_{n,j,1}(t)  &  =(2\pi n+t)^{2}-4\pi nt+(-1)^{j}\sqrt{(4\pi nt)^{2}%
+q_{2n}q_{-2n}},\\
F_{n,j,m+1}(t)  &  =(2\pi n+t)^{2}+\frac{1}{2}(A(F_{n,j,m},t)+A^{^{\prime}%
}(F_{n,j,m},t))-4\pi nt+(-1)^{j}\sqrt{D(F_{n,j,m},t)}%
\end{align*}
for $m=1,2,....$ Moreover we use the functions $A^{\ast}$ , $B^{\ast}$ which
are obtained from $A,$ $B$ respectively by replacing $q_{n_{1}}$ with
$e^{i(2\pi(n-n_{1})+t)x}.$

\begin{theorem}
\textit{ }$(a)$ \textit{If the conditions of Theorem 2 hold, then }the
eigenvalue $\lambda_{n,j}(t)$ satisfies the following , uniform with respect
to $t\in\lbrack0,\rho],$ formulas
\begin{equation}
\lambda_{n,j}(t)=(2\pi n+t)^{2}-4\pi nt+(-1)^{j}\sqrt{(4\pi nt)^{2}%
+q_{2n}q_{-2n}}+O(\frac{1}{n}),
\end{equation}%
\begin{equation}
\lambda_{n,j}(t)=F_{n,j,m}(t)+O(\frac{1}{n^{m}}),\text{ }m=1,2,....
\end{equation}

$(b)$ The normalized eigenfunction $\Psi_{n,j,t}(x)$ corresponding to
$\lambda_{n,j}(t)$ \ is $\frac{\varphi_{n,j,t}(x)}{\parallel\varphi
_{n,j,t}(x)\parallel},$ where $\varphi_{n,j,t}(x)$ satisfies the following,
uniform with respect to $t\in\lbrack0,\rho],$ formulas
\[
\varphi_{n,1,t}(x)=e^{i(-2\pi n+t)x}+\alpha_{n,1}e^{i(2\pi n+t)x}+A^{\ast
}(F_{n,1,m},t)+\alpha_{n,1}B^{\ast}(F_{n,1,m},t)+O(n^{-m-1}),
\]%
\[
\varphi_{n,2,t}(x)=e^{i(2\pi n+t)x}+\alpha_{n,2}e^{i(-2\pi n+t)x}+A^{\ast
}(F_{n,2,m},t)+\alpha_{n,2}B^{\ast}(F_{n,2,m},t)+O(n^{-m-1}),
\]%
\[
\alpha_{n,1}(t)=\frac{C(F_{n,1,m},t)-4\pi nt-\sqrt{D(F_{n,1,m},t)}}%
{q_{-2n}+B^{^{\prime}}(F_{n,1,m},t)}+O(\frac{1}{q_{2n}n^{m+1}})=O(1),
\]%
\[
\alpha_{n,2}(t)=\frac{-C(F_{n,2,m},t)-4\pi nt+\sqrt{D(F_{n,2,m},t)}}%
{q_{2n}+B(F_{n,2,m},t)}+O(\frac{1}{q_{2n}n^{m+1}})=O(1).
\]

\end{theorem}

\begin{proof}
By (73) and (45) to prove (86) it is enough to show that%
\begin{equation}
\sqrt{D(\lambda_{n,j}(t),t)}=\sqrt{(4\pi nt)^{2}+q_{2n}q_{-2n}}+O(\frac{1}%
{n}).
\end{equation}
Using (66) and (69) one can easily verify that%
\[
\mid\sqrt{D(\lambda_{n,j}(t),t)}+\sqrt{(4\pi nt)^{2}+q_{2n}q_{-2n}}\mid
\geq\sqrt{\frac{\varepsilon}{6}}(4\pi nt+\mid\sqrt{q_{2n}q_{-2n}}\mid.
\]
Therefore there exists a constant $c_{4}$ such that
\[
\mid\sqrt{D(\lambda_{n,j}(t),t)}-\sqrt{(4\pi nt)^{2}+q_{2n}q_{-2n}}\mid
=\mid\frac{D_{1}(\lambda_{n,j}(t),t)+D_{2}(\lambda_{n,j}(t),t)}{\sqrt
{D(\lambda_{n,j}(t),t)}+\sqrt{(4\pi nt)^{2}+q_{2n}q_{-2n}}}\mid\leq
\]%
\begin{equation}
c_{4}(\mid\frac{D_{1}(\lambda_{n,j}(t),t)}{4\pi nt}\mid+\mid\frac
{D_{2}(\lambda_{n,j}(t),t)}{\sqrt{q_{2n}q_{-2n}}}\mid).
\end{equation}
Moreover, from (67) and (5) we obtain%
\begin{equation}
\frac{D_{1}(\lambda,t)}{4\pi nt}=O(\frac{1}{n}),\text{ }\frac{D_{2}%
(\lambda,t)}{\sqrt{q_{2n}q_{-2n}}}=o(n^{-s-1}).
\end{equation}
Hence, (88) follows from (89) and (90). \ Thus (86) is proved.

It follows from Lemma 3 and from the proof of (64) that the functions
$A(\lambda,t),$ $A^{^{\prime}}(\lambda,t),$ $B(\lambda,t),$ $B^{^{\prime}%
}(\lambda,t)$ and $\sqrt{D(\lambda,t)}$ satisfy the equality
\begin{equation}
f(F_{n,j,k}(t)+O(n^{-k}),t)=f(F_{n,j,k}(t),t)+O\left(  n^{-k-1}\right)  .
\end{equation}
Now, we prove (87) by induction. It is proved for $m=1$ (see (86) and the
definition of $F_{n,j,1}(t)$ ). Assume that (87) is true for $m=k$.
Substituting the value of $\lambda_{n,j}(t)$ given by (87) for $m=k$ in the
right-hand side of (73) and using (91) we get (87) for $m=k+1.$

$(b)$ Writing the decomposition of the normalized eigenfunction $\Psi
_{n,j,t}(x)$ corresponding to the eigenvalue $\lambda_{n,j}(t)$ by the basis
$\ \{e^{i(2\pi(n-n_{1})+t)x}:n_{1}\in\mathbb{Z}\},$ we obtain
\begin{equation}
\Psi_{n,j,t}(x)-u_{n,j}(t)e^{i(2\pi n+t)x}-v_{n,j}(t)e^{i(-2\pi n+t)x}=
\end{equation}%
\[
\sum_{n_{1}\neq0,2\pi;n_{1}=-\infty}^{\infty}(\Psi_{n,j,t}(x),e^{i(2\pi
(n-n_{1})+t)x})e^{i(2\pi(n-n_{1})+t)x}.
\]
The right-hand side of (92) can be obtained from the right-hand side of (30)
by replacing $q_{n_{1}}$ with $e^{i2\pi(n-n_{1})x}.$ Since (37) is obtained
from (30) by iteration, doing the same, we obtain
\begin{equation}
\Psi_{n,j,t}(x)=u_{n,j}(t)e^{i(2\pi n+t)x}+v_{n,j}(t)e^{i(-2\pi n+t)x}%
+u_{n,j}(t)A^{\ast}(\lambda_{n,j},t)+v_{n,j}(t)B^{\ast}(\lambda_{n,j},t)
\end{equation}
from (92). First, let us consider the case $j=2.$ Using (87) and (91) in (37),
taking into account (46) and (53), we get%
\begin{equation}
\frac{v_{n,2}(t)}{u_{n,2}(t)}=\frac{-C(F_{n,2,m}(t),t)-4\pi nt+\sqrt
{D(F_{n,2,m}(t),t)}}{q_{2n}+B(F_{n,2,m}(t),t)}+O(\frac{1}{q_{2n}n^{m+1}}),
\end{equation}
where $m>s.$ Now, dividing both sides of (93) by $u_{n,2}(t),$ and denoting
\[
\alpha_{n,2}(t)=\frac{v_{n,2}(t)}{u_{n,2}(t)},\text{ }\varphi_{n,2,t}%
(x)=\frac{\Psi_{n,2,t}(x)}{u_{n,2}(t)},
\]
we obtain
\begin{equation}
\varphi_{n,2,t}(x)=e^{i(2\pi n+t)x}+\alpha_{n,2}(t)e^{i(-2\pi n+t)x}+A^{\ast
}(\lambda_{n,2}(t),t)+\alpha_{n,2}(t)B^{\ast}(\lambda_{n,2}(t),t).
\end{equation}
Here $\alpha_{n,2}(t)=O(1)$ due to (80). On the other hand, one can readily
see that the functions $A^{\ast}(\lambda,t)$ and $B^{\ast}(\lambda,t)$ also
satisfy (91). Therefore, from (95) we obtain the proof of $(b)$ for $j=2$. In
the same way, we get the proof of $(b)$ for $j=1$.
\end{proof}

To obtain the asymptotic formulas for the eigenvalues $\lambda_{n,j}(t)$ for
$t\in\lbrack\pi-\rho,\pi]$ instead of (30) we use the formula%
\begin{equation}
(\lambda_{n,j}(t)-(2\pi n+t)^{2})(\Psi_{n,j,t},e^{i(2\pi n+t)x})-q_{2n+1}%
(\Psi_{n,j,t},e^{i(-2\pi(n+1)+t)x})=
\end{equation}%
\[
\sum_{n_{1}\neq0,2n+1;n_{1}=-\infty}^{\infty}q_{n_{1}}(\Psi_{n,j,t}%
,e^{i(2\pi(n-n_{1})+t)x}).
\]
From (30) we obtained (37), (38). In the same way, from (96) we get
\[
(\lambda_{n,j}(t)-(2\pi n+t)^{2}-\widetilde{A}(\lambda_{n,j}(t),t))u_{n,j}%
(t)=(q_{2n+1}+\widetilde{B}(\lambda_{n,j}(t),t))v_{n,j}(t),
\]%
\[
(\lambda_{n,j}(t)-(-2\pi(n+1)+t)^{2}-\widetilde{A}^{^{\prime}}(\lambda
_{n,j}(t),t))v_{n,j}(t)=(q_{-2n-1}+\widetilde{B}^{^{\prime}}(\lambda
_{n,j}(t),t))u_{n,j}(t),
\]
where
\[
\widetilde{A}(\lambda,t)=\sum_{k=1}^{\infty}\widetilde{a}_{k}(\lambda
,t),\text{ }\widetilde{B}=\sum_{k=1}^{\infty}\widetilde{b}_{k},\text{
}\widetilde{A}^{^{\prime}}=\sum_{k=1}^{\infty}\widetilde{a}_{k}^{\prime
},\text{ }\widetilde{B}^{^{\prime}}=\sum_{k=1}^{\infty}\widetilde{b}%
_{k}^{^{\prime}}.
\]
Here $\widetilde{a}_{k},\widetilde{a}_{k}^{\prime},\widetilde{b}%
_{k},\widetilde{b}_{k}^{\prime}$ differ from $a_{k},a_{k}^{\prime},b_{k}%
,b_{k}^{\prime}$ respectively, in the following sense. The sums in
$\widetilde{a}_{k},\widetilde{a}_{k}^{\prime},\widetilde{b}_{k},\widetilde
{b}_{k}^{\prime}$ are taken under conditions $n_{1}+n_{2}+...+n_{s}\neq
0,\pm(2n+1)$ instead of the condition $n_{1}+n_{2}+...+n_{s}\neq0,\pm2n$ for
$s=1,2,...,k.$ Besides in $\widetilde{b}_{k},\widetilde{b}_{k}^{\prime}$ the
multiplicand $q_{\pm2n-n_{1}-n_{2}-...-n_{k}}$ of $b_{k},b_{k}^{\prime}$ is
replaced by $q_{\pm(2n+1)-n_{1}-n_{2}-...-n_{k}}$. Moreover, instead of
$F,\alpha_{n,j},$ $A^{\ast},$ $B^{\ast}$ we use $\widetilde{F},\widetilde
{\alpha}_{n,j},$ $\widetilde{A}^{\ast},$ $\widetilde{B}^{\ast}$ that can be
defined in a similar way. Thus, instead of (5), (53), (62) and (63) using the
relations
\begin{equation}
q_{2n+1}\sim q_{-2n-1},\text{ }\mid q_{2n+1}\mid>cn^{-s-1},\text{ }%
\end{equation}%
\begin{equation}
\operatorname{Re}q_{2n+1}q_{-2n-1}\geq0,
\end{equation}%
\begin{equation}
\mid\operatorname{Im}q_{2n+1}q_{-2n-1}\mid\geq\varepsilon\mid q_{2n+1}%
q_{-2n-1}\mid
\end{equation}
respectively and repeating the proof of Theorems 1-3 we get:

\begin{theorem}
\textit{Let }$q\in W_{1}^{p}[0,1]$\textit{ and (3) holds with some }$s\leq p.$
Suppose (97) and at least one of the inequalities (98)and (99) hold.
Then\textit{ there exists }$N(\pi,\rho)$\textit{ such that the} eigenvalues
$\lambda_{n,j}(t)$ for $n>N(\pi,\rho)$ and $t\in\lbrack\pi-\rho,\pi]$ are
simple and satisfy the formulas
\begin{equation}
\lambda_{n,j}(t)=(2\pi n+t)^{2}-2\pi(2n+1)(t-\pi)+(-1)^{j}\sqrt{(2\pi
(2n+1)(t-\pi))^{2}+q_{2n+1}q_{-2n-1}}+O(\frac{1}{n}),
\end{equation}%
\begin{equation}
\lambda_{n,j}(t)=\widetilde{F}_{n,j,m}(t)+O(n^{-m}),\text{ }m=1,2,....
\end{equation}
The normalized eigenfunction $\Psi_{n,j,t}(x)$ corresponding to $\lambda
_{n,j}(t)$ \ is $\frac{\varphi_{n,j,t}(x)}{\parallel\varphi_{n,j,t}%
(x)\parallel},$ where $\varphi_{n,j,t}(x)$ satisfies the following , uniform
with respect to $t\in\lbrack\pi-\rho,\pi],$ formulas
\[
\varphi_{n,1,t}=e^{i(-2\pi(n+1)+t)x}+\widetilde{\alpha}_{n,1}(t)e^{i(2\pi
n+t)x}+\widetilde{A}^{\ast}(\widetilde{F}_{n,1,m},t)+\widetilde{\alpha}%
_{n,1}(t)\widetilde{B}^{\ast}(\widetilde{F}_{n,1,m},t)+O(n^{-m-1}),
\]%
\[
\varphi_{n,2,t}=e^{i(2\pi n+t)x}+\widetilde{\alpha}_{n,2}(t)e^{i(-2\pi
(n+1)+t)x}+\widetilde{A}^{\ast}(\widetilde{F}_{n,2,m},t)+\alpha_{n,2}%
(t)\widetilde{B}^{\ast}(\widetilde{F}_{n,2,m},t)+O(n^{-m-1}).
\]

\end{theorem}

The following remark follows from Remark 1 and Theorems 1-4.

\begin{remark}
Suppose the conditions of Theorem 2 and Theorem 4 hold. One can readily see
that (86) for $t=0$ and $t=\rho$ gives formulas (10) and (9) respectively, if
we use the notation: $\lambda_{-n}(t)=:\lambda_{n,1}(t)$ for $n=1,2,...$ and
$\,\lambda_{n}(t)=:\lambda_{n,2}(t)$ for $n=0,1,2,...$ Note that \textit{we}
use both notations $\lambda_{n}(t)$ and $\lambda_{n,j}(t)$. If the notation
$\lambda_{n}(t)$ is used, then the corresponding eigenfunction and Fourier
coefficients (see (26)) are denoted by $\Psi_{n,t}(x)$ and $u_{n}(t),$
$v_{n}(t).$ Similarly, (100) for $t=\pi$ and $t=\pi-\rho$ gives the formula
obtained in [1] for $\lambda_{n,j}(\pi)$ and (9). Moreover, there is
one-to-one correspondence between the eigenvalues (counting with
multiplicities) and integers. Indeed, by (9)\textbf{ }and\textbf{ }Theorems
2-4 if $|n|>\max\{N,N(\pi,\rho)\}$ and $t\in\lbrack0,\pi]$ then the
eigenvalues $\lambda_{n}(t)$ and $\lambda_{-n}(t)$ are simple and the number
of the remaining eigenvalues of $L_{t}(q)$\textit{ is} equal to $2\max
\{N,N(\pi,\rho)\}+1,$ where $N$ and $N(\pi,\rho)$ are defined in Remark 1 and
Theorem 4. Further, for simplicity of notation, $\max\{N,N(\pi,\rho)\}$ is
denoted by $N.$ Using the above notation, we see that \textit{the spectrum of
}$L_{t}(q)$\textit{ is }
\begin{equation}
S(L_{t})=\{\,\lambda_{n}(t):n\in\mathbb{Z}\}=\{\lambda_{n,1}%
(t):n=1,2,...\}\cup\{\,\lambda_{n,2}(t)\text{ }n=0,1,2,...\}.
\end{equation}
Since $\lambda_{n}(t)$ for $|n|>N$ is a simple root of
\begin{equation}
F(\lambda)=2\cos t,
\end{equation}
where $F(\lambda)$ is the Hill's discriminant, we have
\begin{equation}
F(\lambda_{n}(t))=2\cos t,\text{ }\frac{dF(\lambda_{n}(t))}{d\lambda}%
\neq0,\text{ }\frac{d\lambda_{n}(t)}{dt}=-(\frac{dF}{d\lambda})^{-1}2\sin t
\end{equation}
for $|n|>N,$ and $t\in\lbrack0,\pi].$ This implies that%
\begin{equation}
\Gamma_{n}=:\{\lambda_{n}(t):t\in\lbrack0,\pi]\}
\end{equation}
is a simple (i.e. $\lambda_{n}:[0,\pi]\rightarrow\Gamma_{n}$ is injective)
analytic arc with endpoints $\lambda_{n}(0)$ and $\lambda_{n}(\pi).$

The eigenvalues of $L_{-t}(q)$ coincides with the eigenvalues of $L_{t}(q),$
because they are roots of equation (103) and $\cos(-t)=\cos t.$ We define the
eigenvalue $\lambda_{n}(-t)$ of $L_{-t}(q)$ by $\lambda_{n}(-t)=\lambda
_{n}(t)$ for all $t\in(0,\pi).$ Then $\lambda_{n}(t)$ is an analytic function
on $(-\pi,\pi]$.
\end{remark}

Using Theorems 2-4 and taking into account Remark 2, we get

\begin{theorem}
\textit{Let }$q\in W_{1}^{p}[0,1]$\textit{ and (3) holds with some }$s\leq p.$
If $q_{n}\sim q_{-n},$ $\mid q_{n}\mid>cn^{-s-1}$ and at least one of the
following inequalities%
\[
\operatorname{Re}q_{n}q_{-n}\geq0,\text{ }\mid\operatorname{Im}q_{n}q_{-n}%
\mid\geq\varepsilon\mid q_{n}q_{-n}\mid
\]
hold for some $c>0$ and $\varepsilon>0$ and for $n>N,$ where $N$ is defined in
Remark 2, then the eigenvalues $\lambda_{n}(t)$ of $L_{t}(q)$ for $\left\vert
n\right\vert >N$ and $t\in\lbrack0,\pi]$ are simple. The eigenvalues
$\lambda_{n}(t)$ and the corresponding eigenfunctions $\Psi_{n,t}(x)$ satisfy
the formulas (9) and the formulas obtained in Theorems 3 and 4.
\end{theorem}

\section{Asymptotic Analysis of L(q)}

Since the spectrum $S(L(q))$ of the operator $L(q)$ is the union of the
spectra $S(L_{t}(q))$ of the operators $L_{t}(q)$ for $t\in\lbrack0,\pi],$ it
follows from (102) and (105) that
\[
S(L(q))=%
{\textstyle\bigcup\limits_{n\in\mathbb{Z}}}
\Gamma_{n}.
\]
By (104) and (105) the subset $\gamma=:\{\lambda_{n}(t):t\in\lbrack
\alpha,\beta]\},$ where $[\alpha,\beta]\subset\lbrack0,\pi],$ of $\Gamma_{n}$
for $|n|>N$ is a regular spectral arc of $L(q)$ in sense of [9] (see
Definition 2.4 of [9]). Following \ [24, 26, 9], we define the projection
$P(\gamma)$ and the spectral singularities as follows
\begin{equation}
P(\gamma)f=\frac{1}{2\pi}%
{\textstyle\int\limits_{\gamma}}
(\Phi_{+}(x,\lambda)F_{-}(\lambda,f)+\Phi_{-}(x,\lambda)F_{+}(\lambda
,f))\frac{\varphi(1,\lambda)}{p(\lambda)}d\lambda,
\end{equation}
where $p(\lambda)=\sqrt{4-F^{2}(\lambda)},$ $F(\lambda)$ is defined in (103),%
\[
\Phi_{\pm}(x,\lambda)=:\theta(x,\lambda)+(\varphi(1,\lambda))^{-1}(e^{\pm
it}-\theta(1,\lambda))\varphi(x,\lambda)
\]
is the Floquet solution and
\[
F_{\pm}(\lambda,f)=\int_{\mathbb{R}}f(x)\Phi_{\pm}(x,\lambda)dx.
\]
The spectral singularities of the operator $L(q)$ are the points of $S(L(q))$
in neighborhoods of which the projections $P(\gamma)$ of the operator $L(q)$
are not uniformly bounded. In other words, we use the following definition.

\begin{definition}
We say that $\lambda\in S(L(q))$ is a spectral singularity of $L(q)$ if for
all sufficiently small $\varepsilon>0$\ there exists a sequence $\{\gamma
_{n}\}$ of the regular spectral arcs $\gamma_{n}\subset\{z\in\mathbb{C}:\mid
z-\lambda\mid<\varepsilon\}$ such that%
\begin{equation}
\lim_{n\rightarrow\infty}\parallel P(\gamma_{n})\parallel=\infty.
\end{equation}

\end{definition}

In the similar way we define the spectral singularity at infinity.

\begin{definition}
We say that the operator $L(q)$ has a spectral singularity at infinity if
there exists a sequence $\{\gamma_{k}\}$ of the regular spectral arcs such
that $d(0,\gamma_{k})\rightarrow\infty$ as $k\rightarrow\infty$ and (107)
holds, where $d(0,\gamma_{k})$ is the distance from the point $(0,0)$ to the
arc $\gamma_{k}.$
\end{definition}

To estimate the projections we use the following lemma of [16]:

\textit{Lemma 5.12 of [16] Let }$A^{^{\prime}}$\textit{ be in }$L_{\infty
}((0,2\pi);B(L_{2}(0,1))).$\textit{ Then for }$f$\textit{ in }$L_{2}%
(-\infty,\infty)$\textit{ the limit in mean }%
\begin{equation}
Af=\lim_{N_{i}\rightarrow\infty}\frac{1}{2\pi}%
{\textstyle\sum\limits_{-N_{1}}^{N_{2}}}
{\textstyle\sum\limits_{-N_{3}}^{N_{4}}}
T_{j}^{\ast}%
{\textstyle\int\limits_{0}^{2\pi}}
e^{it(j-k)}A^{^{\prime}}(t)T_{k}fdt
\end{equation}
\textit{exists and defines a bounded operator in }$L_{2}(-\infty,\infty
)$\textit{ of norm }$\parallel A\parallel\leq\parallel A^{^{\prime}}%
\parallel_{\infty},$\textit{ where }$T_{k}$\textit{ is defined by }%
$T_{k}(f(x))=f(x+k)$\textit{ for }$x\in\lbrack0,1)$\textit{, }$T_{k}%
(f(x))=0$\textit{ for }$x\neq\lbrack0,1)$\textit{ and}

\textit{ }$T_{j}^{\ast}(f(x))=f(x-j)$\textit{ for }$x\in\lbrack j,j+1),$%
\textit{ }$T_{j}^{\ast}(f(x))=0$\textit{ for }$x\neq\lbrack j,j+1).$

Let $\{\chi_{n,t}:n\in\mathbb{Z\}}$ be biorthogonal to $\left\{  \Psi
_{n,t}:n\in\mathbb{Z}\right\}  $ and $\Psi_{n,t}^{\ast}(x)$ be the normalized
eigenfunction of $(L_{t}(q))^{\ast}$ corresponding to $\overline{\lambda
_{n}(t)}.$ Then
\begin{equation}
\chi_{n,t}(x))=\frac{1}{\overline{\alpha_{n}(t)}}\Psi_{n,t}^{\ast}(x),\text{
}\alpha_{n}(t)=(\Psi_{n,t}(x),\Psi_{n,t}^{\ast}(x))_{(0,1)},
\end{equation}
where $(.,.)_{(a,b)}$ denotes the inner product in $L_{2}(a,b).$ One can
easily verify that
\begin{equation}
\Psi_{n,t}(x)=\frac{\Phi_{+}(x,\lambda_{n}(t))}{\mid\Phi_{+}(x,\lambda
_{n}(t))\mid},\text{ }\chi_{n,t}(x))=\frac{1}{\overline{\alpha_{n}(t)}}%
\frac{\overline{\Phi_{-}(x,\lambda_{n}(t))}}{\mid\overline{\Phi_{-}%
(x,\lambda_{n}(t))}\mid},
\end{equation}%
\begin{equation}
\Psi_{n,t}(x+1)=e^{it}\Psi_{n,t}(x),\text{ }\chi_{n,t}(x+1)=e^{it}\chi
_{n,t}(x).
\end{equation}
Now we are ready to prove the main results of this chapter:

\begin{theorem}
If all conditions of Theorem 5 hold, then

$(a)$\ The spectrum of the operator $L(q)$ in a neighborhood of $\infty$
consist of the separated simple analytic arcs $\Gamma_{n}$ for $|n|>N$ with
endpoints $\lambda_{n}(0)$ and $\lambda_{n}(\pi),$ where $N$ is defined in
Remark 2.

$(b)$The operator $L(q)$ has at most finitely many spectral singularities.

$(c)$ The projections $P(\gamma)$ of $L(q)$ for all $\gamma\subset\Gamma_{n}$
and $|n|>N$ are uniformly bounded and the operator $L(q)$ has no spectral
singularity at infinity.
\end{theorem}

\begin{proof}
$(a)$ Due to Remark 2 we need only to note that $\Gamma_{n}$ for $|n|>N$ are
separated, that is, $\Gamma_{n}\cap\Gamma_{k}=\varnothing$ for $k\in
\mathbb{Z}\backslash\{n\}.$ This is true due to the following reason. The
equality $\lambda_{n}(t)=\lambda_{k}(t)$ contradicts the simplicity of
$\lambda_{n}(t).$ The equality $\lambda_{n}(t)=\lambda_{k}(t^{^{\prime}})$ for
$t^{\prime}\neq t$ and $t^{\prime}\in\lbrack0,\pi]$ contradicts the first
equality in (104).

$(b)$ By Theorem 5 the equation $\frac{dF(\lambda)}{d\lambda}=0$ has no zeros
at $\Gamma_{n}$ for $|n|>N.$ Since $\frac{dF(\lambda)}{d\lambda}$ is an entire
function, it has at most finite number of roots on the compact set
$\cup_{|n|\leq N}\Gamma_{n}.$ Now the proof of $(b)$ \ follows from the
well-known fact that the spectral singularities of $L(q)$ are contained in the
set $\{\lambda:\frac{dF(\lambda)}{d\lambda}=0,$ $\lambda\in S(L(q))\}$ (see
[9, 26]).

$(c)$ Changing the variable $\lambda$ to the variable $t$ in the integral in
(106), using%
\[
d\lambda=-p(\lambda)\left(  \frac{dF}{d\lambda}\right)  ^{-1}dt,\text{ }%
\frac{dF(\lambda_{n}(t))}{d\lambda}=-\varphi(1,\lambda_{n}(t))(\Phi
_{+}(x,\lambda_{n}(t)),\overline{\Phi_{-}(x,\lambda_{n}(t))})
\]
and (110) by simple calculations we get%
\begin{equation}
P(\gamma)f(x)=\frac{1}{2\pi}%
{\textstyle\int\limits_{\delta}}
(f,\chi_{n,t})_{\mathbb{R}}\Psi_{n,t}(x))dt,
\end{equation}
where $\delta=\{t\in(-\pi,\pi]:\lambda_{n}(t)\in\gamma\}.$ Let $A^{^{\prime}%
}(t)$ be the operator defined by
\begin{equation}
A^{^{\prime}}(t)f=(f,\chi_{n,t})_{(0,1)}\Psi_{n,t}(x)
\end{equation}
for $t\in\delta$ and $A^{^{\prime}}(t)=0$ for $t\in(-\pi,\pi]\backslash
\delta.$ By (109) we have
\begin{equation}
\parallel A^{^{\prime}}(t)\parallel=\mid\alpha_{n}(t)\mid^{-1},\text{ }\forall
t\in\delta,
\end{equation}
where $\alpha_{n}$ is a continuous function and $\alpha_{n}(t)\neq0$ $,$ since
$\lambda_{n}(t)$ is a simple eigenvalue. Therefore $A^{^{\prime}}\in
L_{\infty}((0,2\pi);B(L_{2}(0,1))).$

Let $f$ $\in C_{0},$ where $C_{0}$ is the set of all compactly supported
continuous functions, and $A$ be the operator defined by (108). Then using
Lemma 5.12 of [16] and (111)-(113) we get
\[
A=\lim_{N_{i}\rightarrow\infty}\frac{1}{2\pi}%
{\textstyle\sum\limits_{j=-N_{1}}^{N_{2}}}
{\textstyle\sum\limits_{k=-N_{3}}^{N_{4}}}
T_{j}^{\ast}%
{\textstyle\int\limits_{0}^{2\pi}}
e^{it(j-k)}A^{^{\prime}}(t)T_{k}f(x)dt=
\]%
\[
\lim_{N_{i}\rightarrow\infty}\frac{1}{2\pi}%
{\textstyle\sum\limits_{j=-N_{1}}^{N_{2}}}
{\textstyle\sum\limits_{k=-N_{3}}^{N_{4}}}
T_{j}^{\ast}%
{\textstyle\int\limits_{0}^{2\pi}}
e^{itj}(f(x+k)e^{-itk},\chi_{n,t})_{(0,1)}\Psi_{n,t}(x)dt=
\]%
\[
\lim_{N_{i}\rightarrow\infty}\frac{1}{2\pi}%
{\textstyle\sum\limits_{j=-N_{1}}^{N_{2}}}
{\textstyle\int\limits_{0}^{2\pi}}
(f,\chi_{n,t})_{\mathbb{R}}e^{itj}T_{j}^{\ast}\Psi_{n,t}(x)dt=P(\gamma)f(x).
\]
Hence $Af=P(\gamma)f$ for all $f$ $\in C_{0}$, where $C_{0}$ is dense in
$L_{2}(-\infty,\infty).$ Moreover, $A$ is bounded by Lemma 5.12 of [16] and
$P(\gamma)$ is bounded, since $\gamma\subset\Gamma_{n}$ and $\Gamma_{n}$ for
$|n|>N$ does not contain the spectral singularities. Therefore, we have
$A=P(\gamma).$ Now Lemma 5.12 of [16] with (114) implies that
\[
\parallel P(\gamma)\parallel\leq\sup_{t\in\delta}\mid\alpha_{n}(t)\mid^{-1}.
\]
On the other hand, by $(a),$ if $\gamma_{k}$ is a regular spectral arcs such
that $d(0,\gamma_{k})$ is a sufficiently large number, then there exists $n$
such that $|n|>N$ and $\gamma_{k}\subset\Gamma_{n}.$ Therefore, $(c)$ follows
from the following lemma.
\end{proof}

\begin{lemma}
If all conditions of Theorem 5 hold, then there exists a constant $d$ such
that%
\begin{equation}
\mid\alpha_{n}(t)\mid^{-1}<d
\end{equation}
for all $\mid n\mid>N$ and $t\in(-\pi,\pi].$
\end{lemma}

\begin{proof}
For $t\in\lbrack\rho,\pi-\rho]$ inequality (115) follows from (9). Now we
prove this for $t\in\lbrack0,\rho].$ The other cases are similar. Since the
boundary condition (2) is self-adjoint we have $(L_{t}(q))^{\ast}=$
$L_{t}(\overline{q}).$ Moreover, the Fourier coefficients of $\overline{q}$
has the form
\[
(\overline{q},e^{i2\pi nx})=\overline{q_{-n}}.
\]
Therefore, one can readily verify that if $q$ satisfies the conditions of
Theorem 5 then $\overline{q}$ also satisfies these conditions. Thus all
formulas and theorems obtained for $L_{t}$ are true for $L_{t}^{\ast}$ if we
replace $q_{n}$ with $\overline{q_{-n}}.$ Hence, by formula (26), we have the
following, uniform with respect to $t\in\lbrack0,\rho],$ formulas
\begin{align}
\Psi_{n,j,t}^{\ast}(x)  &  =u_{n,j}^{\ast}(t)e^{i(2\pi n+t)x}+v_{n,j}^{\ast
}(t)e^{i(-2\pi n+t)x}+h_{n,j,t}^{\ast}(x),\nonumber\\
(u_{n,j}^{\ast}(t))^{2}+(v_{n,j}^{\ast}(t))^{2}  &  =1+O(n^{-1}),(h_{n,j,t}%
^{\ast},e^{i(\pm2\pi n+t)x})=0,\parallel h_{n,j,t}^{\ast}\parallel=O(n^{-1}).
\end{align}
Then%
\begin{equation}
(\Psi_{n,j,t}(x),\Psi_{n,j,t}^{\ast}(x))=u_{n,j}(t)\overline{u_{n,j}^{\ast
}(t)}+v_{n,j}(t)\overline{v_{n,j}^{\ast}(t)}+O(n^{-1}).
\end{equation}
By Lemma 5 we have $v_{n,1}^{\ast}\sim$ $u_{n,2}^{\ast}(t)\sim1$. Using this
and (80) in (117) for $j=1,$ we get
\begin{equation}
(\Psi_{n,1,t},\Psi_{n,1,t}^{\ast})=v_{n,1}(t)\overline{v_{n,1}^{\ast}%
(t)}(1+\frac{u_{n,1}(t)\overline{u_{n,1}^{\ast}(t)}}{v_{n,1}(t)\overline
{v_{n,1}^{\ast}(t)}})+O(n^{-1}).
\end{equation}
It follows from (84) and (85) that
\begin{equation}
\frac{u_{n,1}}{v_{n,1}}=\frac{(q_{2n}+B(t))}{(C(t)-4\pi nt-\sqrt{D(t)})}%
=\frac{-C(t)+4\pi nt-\sqrt{D(t)}}{(q_{-2n}+B^{^{\prime}}(t))}.
\end{equation}
Then $\frac{u_{n,1}^{\ast}(t)}{v_{n,1}^{\ast}(t)}$ satisfies the formula
obtained from (119) by replacing $q_{n}$ with $\overline{q_{-n}}.$ Hence,
$\frac{\overline{u_{n,1}^{\ast}(t)}}{\overline{v_{n,1}^{\ast}(t)}}$ satisfies
the formula obtained from (119) by replacing $q_{n}$ with $q_{-n}.$ Thus, we
have%
\[
\frac{u_{n,1}}{v_{n,1}}\frac{\overline{u_{n,1}^{\ast}(t)}}{\overline
{v_{n,1}^{\ast}(t)}}=\frac{-C(t)+4\pi nt-\sqrt{D}}{(q_{-2n}+B^{^{\prime}}%
(t))}\frac{(q_{-2n}+B^{\ast}(t))}{(C^{\ast}(t)-4\pi nt-\sqrt{D^{\ast}(t)})},
\]
where $B^{\ast},C^{\ast}$ and $D^{\ast}$ are obtained from $B,C$ and $D$ by
replacing $q_{n}$ with $q_{-n}.$ Since
\[
(q_{-2n}+B^{^{\prime}}(t))=q_{-2n}(1+o(1)),\text{ }(q_{-2n}+B^{\ast
}(t))=q_{-2n}(1+o(1))
\]
(see (46), (53)), the last equality can be written in the form%
\begin{equation}
\frac{u_{n,1}}{v_{n,1}}\frac{\overline{u_{n,1}^{\ast}(t)}}{\overline
{v_{n,1}^{\ast}(t)}}=\frac{-C(t)+4\pi nt-\sqrt{D(t)}}{C^{\ast}(t)-4\pi
nt-\sqrt{D^{\ast}(t)}}(1+o(1)).
\end{equation}
Using (57) and (66) for $L_{t}$ and $L_{t}^{\ast}$ one can easily see that%
\[
\mid C(t)\mid+\mid\sqrt{D(t)}\mid+\mid C^{\ast}(t)\mid+\mid\sqrt{D^{\ast}%
(t)}\mid=O(f(n,t)),
\]
where $f(n,t)=\mid4\pi nt\mid+\mid\sqrt{q_{2n}q_{-2n}}\mid.$ Therefore, by
(118), (120) and by $v_{n,1}\sim v_{n,1}^{\ast}\sim1,$ there exists a constant
$c_{5}$ such that
\begin{equation}
\frac{1}{\mid(\Psi_{n,1,t},\Psi_{n,1,t}^{\ast})\mid}<c_{5}\mid\frac{C^{\ast
}(t)-4\pi nt-\sqrt{D^{\ast}(t)}}{C^{\ast}(t)-\sqrt{D^{\ast}(t)}-C(t)-\sqrt
{D(t)}+o(f(n,t))}\mid.
\end{equation}
From (57) and (66) for $L_{t}^{\ast}$ we get
\begin{equation}
\mid C^{\ast}(t)-4\pi nt-\sqrt{D^{\ast}(t)}\mid<\mid9\pi nt\mid+2\mid
\sqrt{q_{2n}q_{-2n}}\mid<3f(n,t).
\end{equation}
Similarly, by (57), (66) and (69) there exists a constant $c_{6}$ such that
\begin{equation}
\mid C^{\ast}(t)-\sqrt{D^{\ast}(t)}-C(t)-\sqrt{D(t)}+o(f(n,t))\mid
>c_{6}f(n,t),
\end{equation}
Thus, using (122) and (123) in (121) we get the proof of the lemma
\end{proof}

It easily follows from this lemma the following result about the asymptotic
spectrality of the operator $L(q)$ defined as follows. Let $e(t,\gamma)$ be
the spectral projection defined by contour integration of the resolvent of
$L_{t}(q)$, where $\gamma\in R$ and $R$ is the ring consisting of all sets
which are the finite union of the half closed rectangles. In [17] it was
proved\ Theorem 3.5 (for the differential operators of arbitrary order with
periodic coefficients rather than for $L(q)$) which, for the operator $L(q),$
can be written in the form:

$L(q)$\textit{ is a spectral operator if and only if }%
\[
\sup_{\gamma\in R}(\sup_{t\in(-\pi,\pi]}\parallel e(t,\gamma)\parallel
)<\infty.
\]
According to this theorem we give the following definition of the asymptotic spectrality.

\begin{definition}
The operator $L(q)$ is said to be an asymptotically spectral operator if there
exists a positive constant $C$ such that
\begin{equation}
\sup_{\gamma\in R(C)}(\sup_{t\in(-\pi,\pi]}\parallel e(t,\gamma)\parallel
)<\infty,
\end{equation}
where $R(C)$ is the ring consisting of all sets which are the finite union of
the half closed rectangles lying in $\{\lambda\in\mathbb{C}:\mid\lambda
\mid>C\}.$
\end{definition}

\begin{theorem}
If all conditions of Theorem 5 hold, then the operator $L(q)$ is an
asymptotically spectral operator in sense of Definition 3.
\end{theorem}

\begin{proof}
Let $C$ be a positive constant such that if $\lambda_{n}(t)\in\{\lambda
\in\mathbb{C}:\mid\lambda\mid>C\},$ then $\mid n\mid>N$ for all $t\in(-\pi
,\pi],$ where $N$ is defined in Remark 2. If $\gamma\in R(C),$ then $\gamma$
contains in a finite number \ of the simple eigenvalues of \ $L_{t}(q).$ Thus,
there exists a finite subset $J(t,\gamma)$ of $\{n\in\mathbb{Z}$: $\mid
n\mid>N\}$ such that the eigenvalue $\lambda_{k}(t)$ lies in $\gamma$ if and
only if $k\in J(t,\gamma).$ It is well-known that these eigenvalues are the
simple poles of the Green function of $L_{t}(q)$ and the projection
$e(t,\gamma)$ has the form
\[
e(t,\gamma)f=\sum_{n\in J(t,\gamma)}\frac{1}{\alpha_{n}(t)}(f,\Psi_{n,t}%
^{\ast})\Psi_{n,t}.
\]
Then by (26) $e(t,\gamma)f$ for $t\in\lbrack0,\rho]$ is the sum of
$e^{+}(t,\gamma)f,$ $e^{-}(t,\gamma)f$ and $e^{h}(t,\gamma)f$, where%
\begin{align*}
e^{+}(t,\gamma)f  &  =\sum_{n\in J(t,\gamma)}\frac{1}{\alpha_{n}(t)}%
(f,\Psi_{n,t}^{\ast})u_{n}(t)e^{i(2\pi n+t)x},\\
e^{-}(t,\gamma)f  &  =\sum_{n\in J(t,\gamma)}\frac{1}{\alpha_{n}(t)}%
(f,\Psi_{n,t}^{\ast})v_{n}(t)e^{i(-2\pi n+t)x},\\
e^{h}(t,\gamma)f  &  =\sum_{n\in J(t,\gamma)}\frac{1}{\alpha_{n}(t)}%
(f,\Psi_{n,t}^{\ast})h_{n,t},\text{ }\parallel h_{n,t}\parallel=O(\frac
{\ln\mid n\mid}{\mid n\mid})
\end{align*}
It follows from (27), (28) and Lemma 6 that there exists a constant $c_{7}$
such that
\begin{equation}
\parallel e^{\pm}(t,\gamma)f\parallel^{2}<2d^{2}\sum_{n:\mid n\mid>N}%
\mid(f,\Psi_{n,t}^{\ast})\mid^{2},
\end{equation}%
\begin{equation}
\parallel e^{h}(t,\gamma)f\parallel<c_{7}d\sum_{n:\mid n\mid>N}\mid
(f,\Psi_{n,t}^{\ast})\mid\frac{\ln\mid n\mid}{\mid n\mid}%
\end{equation}
for all $t\in\lbrack0,\rho]$ and $\gamma\in R(C).$

Now suppose that $\parallel f\parallel=1.$ By (116) there exists a constant
$c_{8}$ such that
\begin{equation}
\sum_{n:\mid n\mid>N}\mid(f,\Psi_{n,t}^{\ast})\mid^{2}\leq c_{8}.
\end{equation}
This inequality with (125) gives
\[
\parallel e^{\pm}(t,\gamma)f\parallel^{2}<2d^{2}c_{8}.
\]
Thus, in (126), first using the Schwarz inequality for $l_{2}$ and then taking
into account (127) we conclude that there exists a positive constant $c_{9}$
such that $\parallel e^{h}(t,\gamma)f\parallel<c_{9}$ and
\begin{equation}
\parallel e(t,\gamma)\parallel<c_{9}%
\end{equation}
for all $t\in\lbrack0,\rho]$ and $\gamma\in R(C).$ In the same way, one can
prove inequality (128) for all $t\in(-\pi,\pi]$ and $\gamma\in R(C),$ that is,
the theorem is proved.
\end{proof}

\end{document}